\documentclass[10pt, reqno]{amsart}
\usepackage{graphicx, amssymb, amsmath, amsthm, color, slashed}
\usepackage{amsmath}
\numberwithin{equation}{section}
\usepackage{enumerate}

\usepackage[numbers]{natbib}
\usepackage{hyperref}
\usepackage[top=2cm, bottom=2cm, left=2cm, right=2cm]{geometry}

\let\Im=\undefined\DeclareMathOperator*{\Im}{Im}

\newcommand{\Real}{\mathbb{R}}
\newtheorem{theorem}{Theorem}[section]

\newtheorem{lemma}[theorem]{Lemma}
\newtheorem{corollary}[theorem]{Corollary}

\newtheorem{proposition}[theorem]{Proposition}

\theoremstyle{definition}
\newtheorem{definition}[theorem]{Definition}
\newtheorem{remark}[theorem]{Remark}

\theoremstyle{remark}

\begin{document}

\title[INHOMOGENEOUS NLS]{ENERGY-CRITICAL INHOMOGENEOUS Nonlinear Schr\"odinger Equation with two
power-type nonlinearities}

\author[Gomes]{Andressa Gomes}
\address{Department of Mathematics, UFPI, Brazil}
\email{gomes.andressa.mat@outlook.com}
\author[Cardoso]{Mykael Cardoso}
\address{ Department of Mathematics, UFPI, Brazil}
\email{mykael@ufpi.edu.br}

\begin{abstract}
	We consider the initial value problem for the inhomogeneous nonlinear Schrödinger equation with double nonlinearities (DINLS)
	\begin{equation*}
			i \partial_t u + \Delta u = \lambda_1 |x|^{-b_1}|u|^{p_1}u + \lambda_2|x|^{-b_2}|u|^{\frac{4-2b_2}{N-2}}u, 
	\end{equation*}
where $\lambda_1,\lambda_2\in \Real$, $3\leq N<6$ and $0<b_1,b_2<\min\{2,\frac{6-N}{2}\}$. In this paper, we establish global well-posedness results for certain parameter regimes and prove finite-time blow-up phenomena under specific conditions. Our analysis relies on stability theory, energy estimates, and virial identities adapted to the DINLS model. 
\end{abstract}

\maketitle

\section{Introduction}
We consider the initial value problem for the inhomogeneous nonlinear Schrödinger equation with double nonlinearities (DINLS)
\begin{equation}\label{IVP2}
	\begin{cases}
		i \partial_t u + \Delta u = \lambda_1 |x|^{-b_1}|u|^{p_1}u + \lambda_2|x|^{-b_2}|u|^{p_2}u, & t \in \mathbb{R}, \ x \in \mathbb{R}^{N}, \\
		u(x,0) = u_0(x) \in H^{1}(\mathbb{R}^{N}),
	\end{cases}
\end{equation}
where \( u(t,x) \) is a complex-valued function defined on spacetime \(\mathbb{R}_t \times \mathbb{R}^{N}_x\) with \(N = 3, 4, 5\). Here, \(\lambda_1\) and \(\lambda_2\) are nonzero real constants, \(0 < p_1 < \frac{4 - b_1}{N - 2}\), \(p_2 = \frac{4 - 2b_2}{N - 2}\), and \(0 < b_1, b_2 < \min \{2, \frac{N}{2} \}\). 

Due to its Hamiltonian structure, the solutions to the problem \eqref{IVP2} have the following conserved quantity in \(H^1(\mathbb{R}^N)\),
\begin{equation}\label{energy}
	E(u(t)) = \int_{\mathbb{R}^N} \left[ \frac{1}{2} |\nabla u(t,x)|^2 + \frac{\lambda_1}{p_1 + 2}|x|^{-b_1}|u(t,x)|^{p_1+2} + \frac{\lambda_2}{p_2 + 2}|x|^{-b_2}|u(t,x)|^{p_2 + 2} \right] \, dx,
\end{equation}
which we refer to as the \textit{Energy}. The solutions of \eqref{IVP2} also conserve the \textit{Mass}, defined by
\begin{equation}\label{mass}
	M(u(t)) = \int_{\mathbb{R}^N} |u(t,x)|^2 \, dx.
\end{equation}

The DINLS extends classical equations such as the nonlinear Schrödinger equation (NLS) and the inhomogeneous nonlinear Schrödinger equation (INLS), with applications in plasma physics, nonlinear optics, and Bose-Einstein condensation. For instance, when \( b_1 = b_2 = \lambda_2 = 0 \), we have the classical nonlinear Schrödinger (NLS) equation,
$$
i\partial_t u + \Delta u + \lambda |u|^{p}u = 0 \,\,\,\, \text{in} \,\, \mathbb{R}^N.
$$
The case \( p = 2 \) is particularly important in physics as it describes the propagation of laser beams in certain types of plasma media (Kerr medium). This equation has been extensively studied (see \cite{sulem1999nonlinear}, \cite{bourgainglobal}, \cite{cazenave03}, \cite{PL}, \cite{taononlinear}, \cite{fibich2015nonlinear}).

When \(\lambda_2 \equiv 0\) and \(0 < b_1 = b < 2\), the equation becomes the inhomogeneous nonlinear Schrödinger (INLS) equation,
\begin{align}\label{INLS}
	i\partial_t u + \Delta u + \lambda |x|^{-b} |u|^{p}u = 0 \quad \text{in} \quad \mathbb{R}^N.
\end{align}
This model describes the propagation of laser beams in nonlinear optics (see \cite{2000Prama} and \cite{liu94}). Initial research on local and global well-posedness of this equation was conducted by Geneoud and Stuart \cite{GS08}, using the abstract theory by Cazenave \cite{cazenave03}. They showed that the initial value problem \eqref{IVP2} is locally well-posed in \(H^1(\mathbb{R}^N)\) if \(0 < p < \frac{2 - b}{N - 2}\) (\(0 < p < \infty\) if \(N = 1, 2\)) and \(0 < b < 2\). Further studies explored well-posedness, stability of standing waves, scattering of global solutions, blow-up of the critical norm, and concentration of the critical norm (see \cite{G12}, \cite{farah16}, \cite{Guz17}, \cite{Dinh21}, \cite{CL19}, \cite{Campos21}, \cite{CFG23}, and \cite{CCF2022}).

When \(b_1 = b_2 = 0\), the equation reduces to the homogeneous double nonlinear Schrödinger (DNLS) equation:
\begin{equation}\label{DNLS}
	i \partial_t u + \Delta u = \lambda_1 |u|^{p_1}u + \lambda_2 |u|^{p_2}u,
\end{equation}
which has received substantial attention recently due to its applications in nonlinear optics and Bose-Einstein condensation (see \cite{Gammal2000}). T. Tao, M. Visan, and X. Zhang in \cite{TVZ2007} studied global well-posedness in the energy space \(H^1(\mathbb{R}^N)\) for \(N \geq 3\) in the mass-subcritical case, \(0 < p_1 < p_2 < \frac{4}{N}\) and \(\lambda_1, \lambda_2 \in \mathbb{R}\), and \(0 < p_1 < p_2 \leq \frac{4}{N-2}\) under \(\lambda_1 \in \mathbb{R}\) and \(\lambda_2 > 0\) (defocusing case). They also obtained scattering results for \(\frac{4}{N} \leq p_1 < p_2 \leq \frac{4}{N-2}\) when \(\lambda_1, \lambda_2 > 0\) or \(\lambda_1 < 0\) (focusing case) and \(\lambda_2 > 0\), by imposing a small mass condition. X. Cheng, C. Miao, and L. Zhao in \cite{CMZ16} established the profile decomposition in \(H^1(\mathbb{R}^N)\) and used the concentration-compactness method to show global well-posedness and scattering versus blow-up in \(H^1(\mathbb{R}^N)\) below the threshold for radial data when \(N = 3, 4\). Y. Luo in \cite{Luo23} addressed the mass-energy double critical case. Other researchers also explored models with two combined nonlinearities (see \cite{BFG23}, \cite{Cao2023}).

Recently, for the inhomogeneous nonlinear Schrödinger equation with double nonlinearities (DINLS):
\begin{equation}
	i \partial_t u + \Delta u = \lambda_1 |x|^{-b_1}|u|^{p_1}u + \lambda_2 |x|^{-b_2}|u|^{p_2}u,
\end{equation}
M. Zhang and J. Zhang \cite{Zhang2024}, in the case \(0=b_1<b_2<\min\{2,N\}\), \(p_1=\frac{4}{N}\) and \(0<p_2<\frac{4-2b_2}{N}\), analyzed blow-up solutions, obtaining sharp thresholds for global existence and blow-up for the Cauchy problem. They proved mass concentration properties, the limiting behavior of blow-up solutions, and proved the nonexistence of minimal mass blow-up solutions. T. Gou, M. Majdoub, and T. Saanouni \cite{GMS2023} studied existence/nonexistence, symmetry, decay, uniqueness, non-degeneracy, and instability of ground states with $0<b_1,b_2<\min\{2,N\}$, $0<p_1<p_2<\frac{4-2b_2}{N-2}$ in low dimension. They proved scattering versus blow-up below the ground state energy threshold using Tao’s scattering criterion and Dodson-Murphy’s Virial/Morawetz inequalities, and provided an upper bound on the blow-up rate.

In this paper, we investigate the case where \(p_2=\frac{4-2b_2}{N-2}\) for \(N \geq 3\), which is the case in which one of the nonlinearities is energy-critical. We present several cases of global well-posedness for \eqref{IVP2}. Additionally, we prove the finite-time blow-up of solutions to \eqref{IVP2} for certain classes of parameters \(\lambda_i\), \(b_i\) (\(i=1,2\)), and the power \(p_1\).

To achieve this, we utilize the \(S^1(L^2, I)\) space (which will be defined rigorously in the next section), characterized as the closure of the test functions under the norm given by the supremum of all Schrödinger-admissible norms. Our main results are described below.

\begin{theorem}\label{gwp}
	Let \(2 <  N < 6\), \(0 < b_1, b_2 < \min\{2 , \frac{6-N}{2}\}\), \(0 < p_1 < \frac{4-2b_1}{N-2}\), and \(p_2 = \frac{4 - 2 b_2}{N-2}\). Then, for all \(u_0 \in H^1(\mathbb{R}^N)\), there exists a unique global solution \(u \in C\left([0,\infty); H^1(\mathbb{R}^N)\right) \cap S^1(L^2, I)\) for the Cauchy problem \eqref{IVP2} in the following cases:
	\begin{enumerate}
		\item \(\lambda_1, \lambda_2 > 0\);
		\item \(\lambda_1 > 0\), \(\lambda_2 < 0\), \(p_1 < p_2\) and \(\frac{p_1}{p_2} b_2 \leq b_1 < b_2\);
		\item \(\lambda_1 > 0\), \(\lambda_2 < 0\), \(b_2 < b_1 < \frac{N(p_2 - p_1)}{p_2 + 2}\).
	\end{enumerate}
	Moreover, for all compact intervals \(I\),
	$$
	\|u\|_{S^1(L^2, I)} \leq C(E, M).
	$$
\end{theorem}

The theorem above follows from arguments involving local well-posedness results and kinetic energy control. To obtain local well-posedness, we view the first nonlinearity \(|u|^{p_1}u\) as a perturbation to the energy-critical INLS. Consequently, we also establish a stability theory for the energy-critical INLS given by \eqref{INLS} with \(p = \frac{4-2b_2}{N-2}\) in \(\dot{H}^1(\mathbb{R}^N)\). In this context, the restriction on the dimension is essential for the continuity argument, which fails in higher dimensions. This strategy was originally applied by \cite{TVZ2007} for the DNLS setting.

The boundedness of the \(\dot{H}^1(\mathbb{R}^N)\)-norm of the solution to \eqref{IVP2} is a consequence of the conserved energy \eqref{energy} and specific embeddings from \(L^2(\mathbb{R}^N) \cap L^{p_2+2}_{b_2}(\mathbb{R}^N)\) to \(L^{p_1+2}_{b_1}(\mathbb{R}^N)\), spaces which will be defined shortly. Here, the signs of the parameters \(\lambda_1\) and \(\lambda_2\), as well as the bounds on \(b_1\) and \(b_2\), are crucial for the necessary control over the kinetic energy.

Now, consider the virial space \(\Sigma\) defined as
\begin{align}
	\Sigma = \{f \in H^1(\mathbb{R}^N)\,  ; \,  |x|f \in L^2(\mathbb{R}^N)\}.
\end{align}
Let \(u_0 \in \Sigma\) and \(u(t)\) be the corresponding solution to \eqref{INLS}. Then, by straightforward computations (see Lemma \ref{BLfinitelemma}), we have that
\begin{align}
	y(t) := \frac{d}{dt} \|x u(t)\|_{L^2}^2 = - \Im \int_{\mathbb{R}^N} \overline{u}(t) \nabla u(t) \cdot x \, dx.
\end{align}

Here, we also present our results on the finite-time blow-up for \eqref{IVP2} with initial data \(u_0 \in \Sigma\). 

\begin{theorem}\label{blwupresult}
	Let \(2 < N < 6\), \(\lambda_2 < 0\), \(0 < b_1, b_2 < \min \{2, \frac{6-N}{2}\}\), \(0 < p_1 < \frac{4-2b_1}{N-2}\), and \(p_2 = \frac{4 - 2 b_2}{N-2}\). Assume that \(u_0 \in \Sigma\), \(y_0 = y(0)\) is positive, and one of the following conditions holds:
	\begin{itemize}
		\item[i)] \(\lambda_1 > 0\), \(p_1 < p_2\), \(b_1 < b_2\), and \(E(u) < 0\); 
		\item[ii)] \(\lambda_1 > 0\), \(b_1 < \frac{N(p_2 - p_1)}{p_2 + 2}\), and \(E(u) < 0\); 
		\item[iii)] \(\lambda_1 < 0\), \(\frac{4-2b_1}{N} < p_1 < p_2\), \(b_1 < b_2\), and \(E(u) < 0\); 
		\item[iv)] \(\lambda_1 < 0\), \(\frac{4-2b_1}{N} < p_1\), \(b_1 < \frac{N(p_2 - p_1)}{p_2 + 2}\), and \(E(u) < 0\); 
		\item[v)] \(\lambda_1 < 0\), \(p_1 < p_2\), \(\frac{p_1 b_2}{p_2} \leq b_1 \leq b_2\), and \(E(u_0) + C M(u_0) < 0\) for some suitably large constant \(C\) (depending on \(N\), \(p_1\), \(p_2\), \(\lambda_1\), and \(\lambda_2\)).
	\end{itemize}
	Then the maximum time \(T^* > 0\) of existence of the solution \(u(t)\) to the problem \eqref{IVP2}, with initial condition \(u(0) = u_0\), is finite in each of the following three cases. In particular, \(0 < T^* \leq C \frac{\|x u_0\|_2^2}{y_0}\) and
	$$
	\lim_{t \rightarrow T^*} \|\nabla u\|_{L^2(\mathbb{R}^N)} = \infty.
	$$
\end{theorem}

Note that in both results above, the restriction on \(p_1\) is similar to results found for DNLS in \cite{TVZ2007} (when \(b_1 = b_2 = 0\)). However, the presence of two inhomogeneous nonlinearities results in a greater restriction on \(b_2\) and consequently on \(b_1\), which differs from the results of \cite{farah16} and \cite{CCF2022}, for instance. In these cases, the parameter \(b\) can be close to zero, while here we have a stronger restriction: \(\frac{p_1}{p_2} b_2 \leq b_1 < b_2\). In this work, we do not cover the high dimension case for DINLS, which remains an open problem.

This paper is organized as follows. Firstly, in Section 2, we discuss Notations and Preliminaries, where we fix the notation used throughout the paper and establish the nonlinear estimates. Next, in Section 3, we develop the stability theory for the energy-critical INLS, i.e., the case \(p_1 = 0\) for \eqref{IVP2}. In Section 4, we prove the local well-posedness for IVP \eqref{IVP2}. In Section 5, we impose conditions on the parameters \(\lambda_1\), \(\lambda_2\), \(p_1\), \(p_2\), \(b_1\), and \(b_2\) to prove Theorem \ref{gwp}. Finally, in Section 6, we establish Theorem \ref{blwupresult}, which ensures the finite-time blow-up of the solutions under certain restrictions on the parameters of \eqref{IVP2}.
\vspace{0.5cm}

\textbf{Acknowledgments.} We thank the Instituto de Matemática Pura e Aplicada (IMPA) for providing the space for our collaborative work. We also thank Luccas Campos from UFMG for his valuable comments and discussions on the interpolation estimates involving the critical potential energy. The authors were partially supported by Fundação de Amparo à Pesquisa do Estado do Piauí - FAPEPI/Brazil and Conselho Nacional de Desenvolvimento Científico e Tecnológico - CNPq/Brazil. 
 \section{Notation and Preliminaries}
\subsection{Notation}

We frequently employ the notation $X \lesssim Y$ to denote the existence of a constant $C$ such that $X \leq CY$, where $C$ may depend on $n$, $p_1$, $p_2$, $\lambda_1$, and $\lambda_2$. Similarly, $X \ll Y$ signifies $X\leq  cY$ for some small constant $c$. The derivative operator $\nabla$ specifically refers to the space variable. Occasionally, we use subscripts to denote spatial derivatives and apply the summation convention over repeated indices.

We use $L_{x}^{r}(\mathbb{R}^{N})$ to denote the Banach space of functions $f: \mathbb{R}^{N} \rightarrow\mathbb{C}$ with the norm
$$
\|f\|_{L_x^{\rho}}:= \left( \int_{\mathbb{R}^{N}} \; |f(x)|^{\rho} \; dx \right)^{\frac{1}{\rho}}
$$
which is finite (with the usual modification when $\rho=\infty$). When $\rho=2$, we simplify $\|f\|_{\rho}$ to $\|f\|$. Similarly, for any space-time slab $I \times \mathbb{R}^{N}$, we use $L_{t}^{\gamma}L_{x}^{\rho}(I \times \mathbb{R}^{N})$ to denote the Banach space of functions $u: I \times \mathbb{R}^{N} \rightarrow \mathbb{C}$ with the norm
$$
\|u\|_{L_{t}^{\gamma}L_{x}^{\rho}(I \times \mathbb{R}^{N})} = \|u\|_{\gamma,\rho}:= \left( \int_{I} \; \|u\|_{\rho}^{\gamma} \; dt \right)^{\frac{1}{\gamma}}
$$
which is finite (with the usual modification when $r=\infty$). We omit mention of the slab $I \times \mathbb{R}^{N}$ when it is clear from context. Additionally, we define the weighted $L^{p}$ norm for any $b>0$ as
$$
\|u\|_{L^{p}_b(\mathbb{R}^{N})} := \left( \int_{\mathbb{R}^{N}} |x|^{-b}|u(x)|^{p} \, dx \right)^{\frac{1}{p}}.
$$
The Fourier transform in the Schwarz space $\mathcal S(\mathbb{R}^{N})$ is defined as
$$
\widehat{f}(\xi) := \int_{\mathbb{R}^{N}} e^{-2\pi i x  \cdot  \xi} f(x) dx.
$$
We make use of the fractional differentiation operators $|\nabla|^{s}$ defined by
$$
\widehat{|\nabla|^{s}f}(\xi) = |\xi|^{s}\widehat{f}(\xi).
$$
These define the homogeneous Sobolev norms
$$
\|f\|_{\dot{H}^{s}_{x}} := \||\nabla|^{s}f\|.
$$

Let $e^{it\Delta}$ denote the free Schrödinger propagator given by 
$$
\widehat{e^{it\Delta}f}(\xi) = e^{-4\pi^2it|\xi|^2}\widehat{f}(\xi).
$$ 
Recall that the solution $u: I \times \mathbb{R}^{N} \rightarrow \mathbb{C}$ to the forced Schrödinger equation  
$$i\partial_tu+ \Delta u = F(x,t)$$ satisfies Duhamel's formula, 
\begin{equation}\label{duhamell}
	u(t) = e^{i(t-t_0)\Delta}u(t_0) - i \int_{t_0}^{t} e^{i(t-\tau)\Delta}F(\cdot, \tau)\, d\tau,
\end{equation}
for any $t_0 \in I$.
\begin{definition}
	Let $N \geq 3$. The pair $(\gamma,\rho)$ is \textit{Schrödinger-admissible} (or $L^{2}$-admissible) if it satisfies the condition
	\begin{equation}\label{Sadm1}
		\frac{2}{\gamma} + \frac{N}{\rho} = \frac{N}{2},
	\end{equation}
	where
	
	\begin{equation}\label{Sadm2}
		2 \leq 
		 \gamma \leq \frac{2N}{N-2}.
	\end{equation}
	The set of all pairs $(\gamma, \rho)$ that are Schrödinger-admissible is denoted by $\mathcal{A}_0$.
\end{definition}

\begin{definition}\label{defS}
	For a time-interval $I \subset \mathbb{R}$, we define the $S(L^{2},I)$ Strichartz norm by 
	$$
	\|u\|_{S(L^{2}, I)} :=\sup_{(\gamma,\rho) \in \mathcal{A}_0} \|u\|_{\gamma,\rho}.
	$$
	We define the associated spaces $S(L^{2}, I)$ as the closure of the test functions under these norms. We denote the dual space of $S(L^{2}, I)$ by $S'(L^{2}, I)$.
\end{definition}

For all $L^{2}$-admissible exponents $(\gamma,\rho)$, it holds
\begin{equation}\label{propS}
	\|u\|_{\gamma,\rho} \lesssim \|u\|_{S(L^{2},I)}
\end{equation}
and consequently
\begin{equation}\label{propdualS}
	\|u\|_{S'(L^{2},I)}  \lesssim \|u\|_{\gamma',\rho'}.
\end{equation}

For a space $X$, we denote $X^1$ as the space
$$
X^1 := \left\{ u : u, \nabla u \in X \right\},
$$
provided with the norm
$$
\|u\|_{X^1} := \|u\|_{X} + \|\nabla u\|_{X}.
$$

Throughout the paper, with $2 <  N < 6$, $0 < b_1 , b_2 < \min\{2, \frac{N}{2}\}$, $0<p_1 <\frac{4-2b_1}{N-2}$ and  $p_2= \frac{4-2b_2}{N-2}$, we define the following Strichartz spaces on $I \times \mathbb{R}^{N}$ as the closure of the test functions under the respective norms:
\begin{equation*}
	\|u\|_{W^{0}(I)}:= \| u\|_{L_{t}^{\frac{2(N+2-2b_2)}{N-2}} L_{x}^{\frac{2N(N+2-2b_2)}{N^{2}+4 -2b_2N}} (I \times \mathbb{R}^{N})}
\end{equation*}
\begin{equation*}
	\|u\|_{V^{0}(I)}:= \|u\|_{L_{t}^{\frac{2(N+2-2b_2)}{N-b_2}} L_{x}^{\frac{2N(N+2-2b_2)}{N^{2}+2b_2 -2b_2N}} (I \times \mathbb{R}^{N})}
\end{equation*}
\begin{equation*}
	\|u\|_{Z(I)}:= \|u\|_{L_{t}^{\frac{2(N+2-2b_2)}{N-2}} L_{x}^{\frac{2N(N+2-2b_2)}{(N-2b_2)(N-2)}} (I \times \mathbb{R}^{N})}.
\end{equation*}
In turn, the spaces $W^{\pm} = W^{\pm}(I \times \mathbb{R}^{N})$ are defined by replacing the norm used in $W$ with the following:
\begin{equation*}
	\|u\|_{W^{\pm}} := \| u\|_{\left[\frac{2(N+2-2b_2)}{N-2}\right]^{\mp}, \left[\frac{2N(N+2-2b_2)}{N^{2}+4 -2b_2N}\right]^{\pm}},
\end{equation*}
where
\begin{equation}\label{exponentpm}
	\frac{1}{\gamma^{\pm}} = \frac{1}{\gamma}  \pm \frac{\eta}{2} \quad \text{and} \quad \frac{1}{\rho^{\pm}} = \frac{1}{\rho} \pm \frac{\eta}{N}
\end{equation}
for $\eta \ll 1$. Finally, we denote
\[
W(I) := W^{0} \cap W^{+}\cap W^{-}.
\]
Similarly, we define the space $V(I)$. It is easy to check that the exponents $(\gamma, \rho)$ which define the space norms $W(I)$ and $V(I)$ are Schrödinger-admissible.

\begin{definition}
	For $N>2$, $0< b< \min \{2, \frac{N}{2}\}$, and $0 < p<\frac{4-2b}{N-2}$, we define the pair
	\begin{equation}\label{expadm1}
		(\gamma_{b}, \rho_{b}) := \left( \frac{4(p+2)}{p(N-2) + 2b} \, , \, \frac{N(p+2)}{N+p - b }\right).
	\end{equation}
	Given an arbitrary slab $I \times \mathbb{R}^{N}$, we define the space
	\[
	X_{b_1, b_2}(I) := L_{t}^{\gamma_{b_1}}L_x^{\rho_{b_1}} (I \times \mathbb{R}^{N}) \cap  L_{t}^{\gamma_{b_1}^+}L_x^{\rho_{b_1}^-} (I \times \mathbb{R}^{N})  \cap  L_{t}^{\gamma_{b_1}^-}L_x^{\rho_{b_1}^+} (I \times \mathbb{R}^{N}) \cap W(I) \cap V(I).
	\]
\end{definition}

\begin{remark}
	Note that the pair $(\gamma_{b_1}, \rho_{b_1})$ defined in \eqref{expadm1} satisfies condition \eqref{Sadm1}. Besides, since $0 < p_1 < \frac{4-2b_1}{N-2}$ and $0<b_1<2$, we have
	\[
	\frac{N-2}{2N} < \frac{1}{\rho_{b_1}}  <  \frac{1}{2}.
	\]
	Thus, $(\gamma_{b_1}, \rho_{b_1})$ is Schrödinger-admissible.
\end{remark}
\subsection{Nonlinear Estimates}
In this section, we establish fundamental nonlinear estimates crucial for our subsequent analysis.

 Let us recall the Strichartz estimates.

\begin{lemma}\label{strichartzest}
	For any admissible pairs $(q,r)$ and $(\gamma ,\rho)$, the following estimates hold:
	\begin{itemize}
		\item[i)] $\|e^{it\Delta}\phi\|_{q,r} \leq C\|\phi\|_{2}$, 
		for any $\phi \in L^2(\mathbb{R}^N)$;
		\item[ii)] $\left\| \displaystyle \int_{t_0}^t e^{i(t-\tau) \Delta} F(\cdot, \tau) \, d\tau \right\|_{q,r} \leq C \|F\|_{\gamma', \rho'}$, uniformly for all $0< T < \infty$ and for any $F \in L_T^{\gamma'}L_x^{\rho'}(\mathbb{R}^N)$.
	\end{itemize}
\end{lemma}

\begin{proof}
	See \cite{cazenave03}, Theorem 2.3.3.
\end{proof}

The above lemma shows that estimating the nonlinear part of \eqref{IVP} is crucial for obtaining the results in this work. This is addressed in the next lemma, whose proof is inspired by Lemma 2.8 in \cite{CCF2022} and Lemma 2.5 in \cite{TVZ2007}. Special care was required in the norms considered due to the auxiliary spaces involved.

 \begin{lemma}\label{lemmaf}
 	Let $I$ be a compact time interval, $N > 2$, $0 < b_1, b_2 < \min \{\frac{N}{2}, 2 \}$, $0 < p_1 < \frac{4-2b_1}{N-2}$, and $p_2 = \frac{4-2b_2}{N-2}$. Then the following inequalities hold
 	\begin{equation}\label{estnl9} 
 		\||x|^{-b_i} |u|^{p_i}\nabla v\|_{S'(L^{2}, I)} \lesssim |I|^{1- \frac{p_i +2 }{\gamma_{b_i}}}\| \nabla u\|^{p_i}_{X_{b_1, b_2}(I)}\| \nabla v\|_{X_{b_1, b_2}(I)},\,\,\,\,\,\mbox{ for }i=1,2,
 	\end{equation}
 	\begin{equation}\label{estnl9'} 
 		\||x|^{-(b_i+1)} |u|^{p_i}  v\|_{S'(L^{2}, I)} \lesssim |I|^{1- \frac{p_i +2 }{\gamma_{b_i}}}\| \nabla u\|^{p_i}_{X_{b_1, b_2}(I)}\| \nabla v\|_{X_{b_1, b_2}(I)},\,\,\,\,\,\mbox{ for }i=1,2,
 	\end{equation}
 	\begin{equation}\label{estnl10}
 		\| |x|^{-b_i} |u|^{p_i} v \|_{S'(L^{2}, I)} \lesssim |I|^{1- \frac{p_i+2}{\gamma_{b_i}}}\| \nabla u\|^{p_i}_{X_{b_1 , b_2}(I)} \| v\|_{X_{b_1 , b_2}(I)},\,\,\,\,\,\mbox{ for }i=1,2.
 	\end{equation}
 \end{lemma}
 
 \begin{proof}
Consider $(\gamma, \rho) \in \mathcal{A}_{0}$, and let $\gamma'$ and $\rho'$ be the Hölder conjugates of $\gamma$ and $\rho$, respectively. For any $0 < \eta \ll 1$, we define
\begin{align}
	R = \left(\frac{\|\nabla v\|_{\rho_-}}{\|\nabla v\|_{\rho_+}}\right)^{\frac{1}{2\eta}},
\end{align}
where $\frac{1}{\rho_\pm} = \frac{1}{\rho} \pm \frac{\eta}{N}$. Since $N > 2$ and $b_i < \frac{N}{2}$, it follows that $b_i < N$. Consequently, if $B_R$ denotes the ball of radius $R > 0$ centered at the origin, we have
\begin{align*}
	\|x|^{-b_i}\|_{L^{\frac{N}{b_i + \eta}}(B_R)} &= R^\eta, \\
	\|x|^{-b_i}\|_{L^{\frac{N}{b_i - \eta}}(B_R^{c})} &= R^{-\eta}.
\end{align*}

Using Hölder's inequality and Sobolev embedding, we obtain for $i=1,2$:
\begin{equation*}
	\begin{split}
		\||x|^{-b_i} |u|^{p_i}\nabla v\|_{L_x^{\rho'}} &\lesssim \||x|^{-b_i}\|_{L_x^{\frac{N}{b_i + \eta}}(A)} \||u|^{p_i}\nabla v\|_{L_x^{\beta^{-}_{1}}} \\
		&\quad + \||x|^{-b_i}\|_{L_x^{\frac{N}{b_i - \eta}}(A^{c})} \||u|^{p_i}\nabla v\|_{L_x^{\beta^{+}_{1}}} \\
		&\lesssim \| \nabla u\|^{p_i}_{L_x^{\beta}}(R^{\eta}\|\nabla v\|_{\rho^+} + R^{-\eta}\|\nabla v\|_{\rho^-}) \\
		&= \| \nabla u\|^{p_i}_{L_x^{\beta}}(\|\nabla v\|_{\rho^+}\|\nabla v\|_{\rho^-})^\frac{1}{2},
	\end{split}
\end{equation*}
where $\frac{1}{\rho'} = \frac{b_i}{N} + \frac{1}{\beta}$, $\frac{1}{\beta_\pm} = \frac{1}{ \beta} \pm \frac{\eta}{N}$, and
\begin{equation}\label{cond11}
	\frac{1}{\rho '} = \frac{b_i}{N} + \frac{p_i}{\beta} - \frac{p_i}{N} + \frac{1}{\rho}.
\end{equation} 
 Again, using Hölder's inequality in the inequality above, we obtain
 \begin{equation}\label{nonlinearestimate11}
 	\| |x|^{-b_i} |u|^{p_i} \nabla v \|_{\gamma', \rho'} \lesssim |I|^{1 - \frac{p_i + 2}{\gamma_{b_i}}} \| \nabla u \|_{\alpha, \beta}^{p_i} \left[ \| \nabla v \|_{\gamma^{-}, \rho^{+}} \| \nabla v \|_{\gamma^+, \rho^-} \right]^{\frac{1}{2}},
 \end{equation}
 where $\frac{1}{\gamma^\pm} = \frac{1}{\gamma} \pm \frac{\eta}{2}$ and
 \begin{equation}\label{cond12}
 	\frac{1}{\gamma'} = \frac{p_i}{\alpha} + \frac{1}{\gamma} + 1 - \frac{p_i + 2}{\gamma_{b_i}}.
 \end{equation}
 
 Firstly, consider $i=1$ in \eqref{nonlinearestimate11}. In this case, we take $(\alpha, \beta) = (\gamma, \rho)$ in \eqref{cond11} and \eqref{cond12}, and thus, $(\alpha, \beta) = (\gamma_{b_1}, \rho_{b_1})$, where $(\gamma_{b_1}, \rho_{b_1})$ is defined as in \eqref{expadm1}. Therefore,
 \begin{equation*}
 	\begin{split}
 		\| |x|^{-b_i} |u|^{p_1} \nabla v \|_{\gamma', \rho'} &\lesssim |I|^{1 - \frac{p_1 + 2}{\gamma_{b_1}}} \| \nabla u \|_{\gamma_{b_1}, \rho_{b_1}}^{p_1} \left[ \| \nabla v \|_{\gamma_{b_1}^{+}, \rho_{b_1}^{-}} \| \nabla v \|_{\gamma_{b_1}^{-}, \rho_{b_1}^{+}} \right]^{\frac{1}{2}}.
 	\end{split}
 \end{equation*}
 The inequality above shows \eqref{estnl9} for $i=1$.
 
 Now, consider $i=2$ in \eqref{nonlinearestimate11}. Here, we put $(\alpha, \beta) = \left(\frac{2(N+2-2b_2)}{N-2}, \frac{2N(N+2-2b_2)}{N^2 + 4 - 2b_2N}\right)$. Hence, from \eqref{cond12} we obtain
 \begin{equation*}
 	\frac{1}{\gamma} = \frac{1}{2} - \frac{p_2}{2} \frac{N-2}{2(N+2-2b_2)} = \frac{N-b_2}{2(N+2-2b_2)}.
 \end{equation*}
 Since $(\gamma, \rho)$ is Schrödinger-admissible, we have $\rho = \frac{2N(N+2-2b_2)}{N^2 + 2b_2 - 2b_2N}$. It is easy to see that for these choices, the equation  \eqref{cond11} is satisfied. So,
 \begin{equation*}
 	\| |x|^{-b_2} |u|^{p_2} \nabla v \|_{S'(L^{2}, I)} \lesssim \| |x|^{-b_2} |u|^{p_2} \nabla v \|_{\gamma', \rho'} \lesssim \| \nabla u \|_{W^{0}(I)}^{p_2} \left[ \| \nabla v \|_{V^{+}(I)} \| \nabla v \|_{V^{-}(I)} \right]^{\frac{1}{2}}.
 \end{equation*}
 The last inequality concludes the proof of \eqref{estnl9}. The proof of \eqref{estnl9'} and \eqref{estnl9'} and \eqref{estnl10} are analogous.
 \end{proof}

\begin{corollary}\label{nlestimates2}
 Let $I$ be a compact time interval, $N > 2$, $0 < b_1, b _2 < \min \{\frac{N}{2}, 2 \}$,  $0 < p_1 < \frac{4-2b_1}{N-2}$ and $p_2 = \frac{4-2b_2}{N-2} $ and $\lambda_1$, $\lambda_2$ be nonzero real numbers. Then,
 \begin{equation*}
 \|\lambda_1|x|^{-b_1}|u|^{p_1}u + \lambda_2|x|^{-b_2}|u|^{p_2}u\|_{S'(L^{2},I)} \lesssim \displaystyle\sum_{i=1}^{2}\; |I|^{1- \frac{p_i+2}{\gamma_{b_i}}}\| \nabla u\|_{X_{b_1, b_2}(I)}^{p_i} \| u\|_{X_{b_1 , b_2}(I)},
 \end{equation*}
  \begin{equation*}
 \|\nabla \left(\lambda_1|x|^{-b_1}|u|^{p_1}u + \lambda_2|x|^{-b_2}|u|^{p_2}u\right)\|_{S'(L^{2},I)} \lesssim \displaystyle\sum_{i=1}^{2}\; |I|^{1- \frac{p_i+2}{\gamma_{b_i}}}\| \nabla u\|_{X_{b_1, b_2}(I)}^{p_i} \|\nabla u\|_{X_{b_1 , b_2}(I)}
 \end{equation*}
 and
  \begin{equation*}
  \begin{split}
 \| \left(\lambda_1|x|^{-b_1}|u|^{p_1}u + \right. & \left.   \lambda_2|x|^{-b_2}|u|^{p_2}u \right) - \left(\lambda_1|x|^{-b_1}|v|^{p_1}v + \lambda_2|x|^{-b_2}|v|^{p_2}v \right) \|_{S'(L^{2},I)} \\
 \lesssim & \displaystyle\sum_{i=1}^{2}\; |I|^{1- \frac{p_1+2}{\gamma_{b_1}}} \left( \|\nabla u\|_{X_{b_1, b_2}(I)}^{p_i}  + \| \nabla v\|_{X_{b_1 , b_2}(I)}^{p_i} \right)  \|u -v\|_{X_{b_1 , b_2}(I)}
 \end{split}
 \end{equation*}
\end{corollary}

\begin{proof}
	It is known that if \( f(u) = |u|^{p}u \), then by the chain rule, we have
	\[
	\nabla (|x|^{-b}|u|^{p}u) = |x|^{-(b+1)} f(u) + |x|^{-b}f_{z}(u) \nabla u + |x|^{-b} f_{\overline{z}} \overline{\nabla u}
	\]
	where \( f_{z} \) and \( f_{\overline{z}} \) are the usual complex derivatives:
	\[
	f_{z} (u) = \frac{p+2}{2}|u|^{p} \quad \text{and} \quad f_{\overline{z}} (u) = \frac{p}{2}|u|^{p-2}u^2.
	\]
	Thus,
	\[
	|\nabla \left( |x|^{-b}|u|^{p}u \right)| \lesssim |x|^{-(b+1)}|u|^{p+1} + |x|^{-b}|u|^{p} |\nabla u|.
	\]
	In addition, the mean value theorem ensures that
	\[
	f(u) - f(v) = \int_{0}^{1} \left[ f_{z}(v+ \theta (u-v))(u-v) + f_{\overline{z}} (v + \theta (u-v)) \overline{(u-v)} \right] \, d\theta 
	\]
	for any \( u, v \in \mathbb{C} \), which in turn implies that
	\[
	|f(u) - f(v)| \lesssim  \left( |u|^{p} + |v|^{p} \right) |u-v| = |u|^{p} |u-v|  + |v|^{p} |u-v|.
	\]
	Thus,
	\[
	\left| |x|^{-b} |u|^{p} u - |x|^{-b} |v|^{p} v \right| \lesssim |x|^{-b} \left( |u|^{p} + |v|^{p} \right) |u-v|.
	\]
	In this way, the proof of the corollary follows from the above note and Lemma \ref{lemmaf}.
\end{proof}

\begin{remark}\label{remarkcont}
	From the proof of Lemma \ref{lemmaf} and Corollary \ref{nlestimates2}, we can derive a more precise estimate for the case when \( i = 2 \). Specifically, we have
	\begin{equation*}
		\begin{split}
			\| |x|^{-b_2}|u|^{p_2}u - |x|^{-b_2}|v|^{p_2}v \|_{\gamma', \rho'} 
			\lesssim & \left( \| \nabla u \|_{W(I)}^{p_2}  + \| \nabla v \|_{W(I)}^{p_2} \right) \| u - v \|_{V(I)},
		\end{split}
	\end{equation*}
	where \((\gamma', \rho')\) are the dual exponents of \(\left(\frac{2(N+2-2b_2)}{N-b_2}, \frac{2N(N+2-2b_2)}{N^2 + 2b_2 - 2b_2N}\right)\).
\end{remark}

\begin{corollary}\label{cor1}
Let $N > 2$ and $0 < b_2 < \min \{ \frac{N}{2} , 2 \}$. Then the following inequality hold
\begin{equation}\label{nonlinearestimate3}
\||x|^{-b_2}|u|^{p_2-1}w \nabla v\|_{2, \frac{2N}{N+2}} \lesssim \| \nabla  u\|_{W(I)}^{p_2-1} \| \nabla  w\|_{W(I)}  \| \nabla v\|_{W(I)}
\end{equation}
and
\begin{equation}\label{nonlinearestimate3'}
\||x|^{-(b_2+1)}|u|^{p_2-1} w  v\|_{2, \frac{2N}{N+2}} \lesssim \| \nabla  u\|_{W(I)}^{p_2-1} \| \nabla  w\|_{W(I)} \| \nabla v\|_{W(I)}.
\end{equation}
In particular,
\begin{equation}\label{nonlinearestimate4}
\|\nabla \left(|x|^{-b_2}|u|^{p_2}u \right)\|_{2, \frac{2N }{N+2}} \lesssim \|\nabla  u\|_{W(I)}^{p_2+1}.
\end{equation}
\end{corollary}

\begin{proof}
We will show the inequality \eqref{nonlinearestimate3}. The other inequalities are proven in similar way.  Just like in Lemma \ref{lemmaf} we obtain
\begin{equation*}
\||x|^{-b_i}|u|^{p_2-1} w \nabla v \|_{\gamma' , \rho'}   \lesssim   \| \nabla u\|_{\alpha, \beta}^{p_2-1} \| \nabla w\|_{\alpha, \beta}\left[\|\nabla v\|_{\alpha^{+},\beta^{-}} \|\nabla v\|_{\alpha^{-} , \beta^{+}} \right]^{\frac{1}{2}},
\end{equation*}
with $(\gamma , \rho)$ and $(\alpha, \beta)$ satisfying 
\begin{align}\label{cond21}
		\frac{1}{\rho '} = \frac{b_i}{N} + \frac{p_2}{\beta} - \frac{p_2}{N} + \frac{1}{\beta}
\end{align}
and 
\begin{align}\label{cond22}
		\frac{1}{\gamma'} = \frac{p_2+1}{\alpha}.
\end{align}
Thus, the inequality holds taking $$(\alpha , \beta) =\left(\frac{2(N+2 -2b_2)}{N-2} , \frac{2N(N +2 -2b_2)}{N^2 + 4 - 2b_2N} \right)$$ and $(\gamma', \rho') = \left(2, \frac{2N}{N+2}\right)$ in \eqref{cond21} and \eqref{cond22}. 
\end{proof}

\begin{corollary}\label{cor2}
	Let \( N > 2 \) and \( 0 < b_2 < \min \left\{ \frac{N}{2}, 2 \right\} \). Then the following inequalities hold:
	\begin{equation}\label{nonlinearestimate5}
		\||x|^{-b_2}|u|^{p_2} \nabla v\|_{2, \frac{2N}{N+2}} \lesssim \| u \|_{Z(I)}^{p_2} \| \nabla v \|_{W(I)},
	\end{equation}
	and 
	\begin{equation}\label{nonlinearestimate5'}
		\||x|^{-(b_2+1)}|u|^{p_2} v\|_{2, \frac{2N}{N+2}} \lesssim \| u \|_{Z(I)}^{p_2} \| \nabla v \|_{W(I)}.
	\end{equation}
	In particular,
	\begin{equation}\label{nonlinearestimate6}
		\|\nabla \left( |x|^{-b_2}|u|^{p_2}u \right)\|_{2, \frac{2N}{N+2}} \lesssim \| u \|_{Z(I)}^{p_2} \| \nabla u \|_{W(I)}.
	\end{equation}
\end{corollary}

\begin{proof}
	As in the proof of \eqref{nonlinearestimate11}, we have
	\begin{equation*}
		\||x|^{-b_2}|u|^{p_2} \nabla v\|_{2, \frac{2N}{N+2}}, \, \||x|^{-(b_2+1)}|u|^{p_2} v\|_{2, \frac{2N}{N+2}} \lesssim \|u\|_{\alpha, \beta}^{p_2} \left( \|\nabla v\|_{W^{+}_b}\| \nabla v\|_{W^{-}_b} \right)^{\frac{1}{2}},
	\end{equation*}
	where
	\begin{equation*}
		\begin{cases}
			\frac{p_2}{\beta} = \frac{N+2}{2N} - \frac{b_2}{N} - \frac{N^2 + 4 - 2b_2N}{2N(N+2-2b_2)} = \frac{(N-2b_2)(4-2b_2)}{2N(N+2-2b_2)}, \\
			\frac{p_2}{\alpha} = \frac{1}{2} - \frac{N-2}{2(N+2-2b_2)} = \frac{4-2b_2}{2(N+2-2b_2)}.
		\end{cases}
	\end{equation*}
	By the definition of the spaces \( Z(I) \), the result follows.
\end{proof}
\begin{remark}\label{remark2}
	From Lemma \ref{lemmaf} and Lemma \ref{propS}, if \( N > 2 \), \( 0 < b_1, b_2 < \min \left\{ \frac{N}{2}, 2 \right\} \), \( 0 < p_1 < \frac{4-2b_1}{N-2} \), and \( p_2 = \frac{4-2b_2}{N-2} \), it is true that
	\begin{equation}\label{estnl8}
		\|\nabla (|x|^{-b_i}|u|^{p_i} u)\|_{S'(L^{2}, I)} \lesssim |I|^{1 - \frac{p_i + 2}{\gamma_{b_i}}} \| \nabla u \|_{S(L^{2}, I)}^{p_i + 1},
	\end{equation}
	and if \( i = 2 \), the power of \( |I| \) is zero.
\end{remark}

 \section{Energy-critical INLS}

 In this section we study  the initial value problem (IVP)
\begin{equation}\label{IVP}
 \begin{cases}
 i \partial_t u + \Delta u =  |x|^{-b}|u|^{\frac{4-2b}{N-2}}u  \ \ \ t \in \mathbb{R}, \ x \in \mathbb{R}^{N},\\
 u(x,0) = u_0(x) \in H^{1}(\mathbb{R}^{N}).
 \end{cases}
 \end{equation}
 where $u(x,t)$ is a complex-valued function and $N >2$. Here we will omit the index $b$ in the spaces $W(I)$, $Z(I)$ and its derivatives. 
 
 \subsection{Well-posedness}
 
In this section, we establish key results regarding the well-posedness of the energy-critical Inhomogeneous Nonlinear Schrödinger Equation (INLS) in the energy space $H^1_x(\mathbb{R}^N)$. Our inspiration comes from the interesting work of Tao and Visan \cite{TV2005}. We have adapted their ideas and studied the well-posedness of the energy-critical problem for the INLS. Specifically, we establish both local and global well-posedness, followed by results related to the existence of solutions to the approximation problem. We start with the local well-posedness.
\begin{theorem}[Local Well-Posedness]\label{lwpECnls}
	Let $I$ be a compact time interval containing $t_0$. Suppose $u_0 \in H^{1}_{x}$ satisfies
	\begin{equation}\label{data1}
		\|e^{i(t-t_0)\Delta} \nabla u_0\|_{W(I)} \leq \eta,
	\end{equation}
	for some $0 < \eta \leq \eta_0$, where $\eta_0 > 0$ is a small constant. There exists a unique solution $w$ to \eqref{IVP} satisfying
	\begin{equation}\label{lwp1}
		\|\nabla u\|_{W(I)} \lesssim \eta,
	\end{equation}
	\begin{equation}\label{lwp2}
		\|\nabla u\|_{S(L^{2},I)} \lesssim \|\nabla u_0\|_{H^{1}} + \eta^{\frac{N+2 - 2b}{N-2}},
	\end{equation}
	and
	\begin{equation}\label{lwp3}
		\|u\|_{S(L^{2},I)} \lesssim \|u_0\|_{L^{2}}.
	\end{equation}
	Moreover, the maps associating the data $u_0 \in H^{1}_{x}$ with the solution $w$ satisfying \eqref{lwp1}-\eqref{lwp3} are Lipschitz.
\end{theorem}
 
 \begin{proof}
\textbf{Existence.} We proceed iteratively. For all $t$, we define the following iterates:
\begin{align*}
	u^{(0)}(t) &:= 0, \\
	u^{(1)}(t) &:= e^{i(t-t_0)\Delta}u_0,
\end{align*}
and for $m \geq 1$,
\begin{equation}\label{solite}
	u^{(m+1)}(t) := e^{i(t-t_0)\Delta}u_0 - i \int_{t_0}^{t} e^{i(t - \tau)\Delta}F(u^{(m)}(\tau))\; d\tau,
\end{equation}
where $F(u) = |x|^{-b}|u|^{p}u$.

Let us first remark that the smallness of the free evolution implies that all the iterates are small. Indeed, by Strichartz estimates and the triangle inequality, followed by Corollary \ref{cor1} and \eqref{data1}, we have
\begin{equation}\label{estite}
	\| \nabla u^{(m + 1)}\|_{W(I)} \lesssim \| \nabla e^{i(t-t_0)\Delta}u_0 \|_{W(I)} + \|\nabla F(u^{(m)})\|_{2, \frac{2N}{N+2}} \lesssim \eta + \|\nabla u^{(m)}\|_{W(I)}^{p+1}.
\end{equation}

Using \eqref{data1} as the base case for an induction hypothesis and choosing $\eta_0$ sufficiently small, by continuous argument we deduce
\begin{equation}\label{estite2}
\| \nabla u^{(m + 1)}\|_{W(I)} \lesssim \eta \ , \	\| \nabla u^{(m)}\|_{S(L^{2}, I)} \lesssim \|u_0\|_{\dot{H}^{1}(\mathbb{R}^{N})} + \eta^{p+1}
\end{equation}
for all $m \geq 1$.

Next, we consider differences of the form $u^{(m+1)} - u^{(m)}$. By the recurrence relation \eqref{solite}, in order to estimate $u^{(m+1)} - u^{(m)}$, we need to control $F(u^{(m)}) - F(u^{(m-1)})$. By Lemma \ref{strichartzest}, Remark \ref{remarkcont} (with $u = u^{(m)}$ and $v = u^{(m-1)}$) and \eqref{estite2}, we estimate 
\begin{equation*}
	\begin{split}
		\|u^{(m+1)} - u^{(m)}\|_{S(L^{2}, I)} \lesssim & \|F(u^{(m)}) - F(u^{(m-1)})\|_{\gamma', \rho'}\\ \lesssim & \|u^{(m)} - u^{(m-1)}\|_{V(I)} \left( \|\nabla u^{(m)}\|_{W(I)}^{p} + \|\nabla u^{(m-1)}\|_{W(I)}^{p}  \right)\\
		\lesssim & \, \eta^{p} \, \|u^{(m)} - u^{(m-1)}\|_{S(L^2, I)}
	\end{split}
\end{equation*}
for all $m \geq 1$, where $(\gamma', \rho')$ is the dual exponent of $\left(\frac{2(N+2-2b_2)}{N-b_2}, \frac{2N(N+2-2b_2)}{N^2 +2b_2 -2b_2N}\right)$. Also, since $u_0 \in L^{2}(\mathbb{R}^{N})$ by hypothesis, we have
$$
\|u^{(1)}\|_{S(L^2,I)} \lesssim \|u_0\|_{L^{2}_x}
$$
due to Lemma \ref{strichartzest}.

Choosing $\eta$ sufficiently small, we conclude that the sequence $\{ u^{(m)} \}_m$ is Cauchy in $S(L^2, I)$. Hence, there exists $u \in S(L^2, I)$ such that $u^{(m)}$ converges to $u$ in $S(L^2, I)$ (and thus also in $V(I)$) as $m \rightarrow 1$. Furthermore, for $m$, $m_0$ large enough (with $m_0$ fixed), we have
\begin{equation*}
	\begin{split}
		\|u\|_{S(L^2, I)}  \leq &  \|u - u^{(m+1)}\| _{S(L^2, I)} +  \| u^{(m+1)} -  u^{(m_0)} \| _{S(L^2, I)} + \| u^{(m_0)} \| _{S(L^2, I)}  \\
		\lesssim &  \sum_{j=0}^{m_0} \|u^{(j+1)} - u^{(j)} \| _{S(L^2, I)} \lesssim C(\eta)  \|u^{(1)} \| _{S(L^2, I)} \lesssim \|u_0\|_{L^2_x},
	\end{split}
\end{equation*}
which ensures \eqref{lwp3}.

The above argument also shows that $F(u^{(m)})$ converges to $F(u)$ in $L^{\gamma'}_t L^{\rho'}_x (I \times \mathbb{R}^N)$, where $(\gamma', \rho')$ is the dual exponent of $\left(\frac{2(N+2-2b_2)}{N-b_2}, \frac{2N(N+2-2b_2)}{N^2 +2b_2 -2b_2N}\right)$. In turn, from Lemma \eqref{strichartzest},
\begin{equation*}
	\begin{split}
		\left\|u - e^{i(t-t_0)\Delta}u_0 + i \int_{t_0}^t e^{i(t-\tau)\Delta} F(u(\cdot, \tau))\; d \tau\right\|_{S(L^{2}, I)} \leq & \|u -u^{(m+1)}\|_{S(L^{2}, I)} \\
		& + \left\|  \int_{t_0}^t e^{i(t-\tau)\Delta} \left[ F(u) -F(u^{(m)}) \right] (\cdot, \tau) \; d \tau\right\|_{S(L^{2}, I)}\\
		\lesssim &  \|u -u^{(m+1)}\|_{S(L^{2}, I)} +  \|F(u) -F(u^{(m+1)})\|_{\gamma' , \rho'}.
\end{split}
\end{equation*}
This implies that $u$ satisfies the equation \eqref{duhamell}.

Additionally, as $u^{(m)}$ converges strongly to $u$ in $S(L^{2},I)$ and given that $\|\nabla u\|_{W(I)}$ and $\|\nabla u \|_{S(L^2,I)}$ are bounded according to \eqref{estite} and \eqref{estite2}, we conclude that $\nabla u^{(m)}$ converges weakly to $\nabla u$ in ${W}(I)$ and ${S}(L^{2},I)$, obtaining the bounds \eqref{lwp1} and \eqref{lwp2}. Furthermore, since $u$ belongs to $S^1(L^{2},I)$, it also lies in $C([0,T], H^{1}(\mathbb{R}^{N}))$ and serves as a solution to \eqref{duhamell}.

\vspace{0.5cm}
\textbf{Uniqueness.} Let $I$ be a time interval containing $t_0$, and let $u_1, u_2$ be two strong solutions to \eqref{IVP}, meaning that $\nabla u_1$, $\nabla u_2 \in S(L^2,I)$ and satisfying the Duhamell formulation \eqref{duhamell}. By a standard argument, we may shrink $I$ so that we can construct the solution given by the preceding iteration argument. Without loss of generality, we may take $u_2$ to be this solution. Thus, $u_2$ satisfies \eqref{lwp1}-\eqref{lwp3}.

Write $v := u_1 - u_2 \in S(L^2, I)$. It is sufficient to show that $v$ vanishes almost everywhere on $[t_0, t_0 + \tau) \subset I$ for some $\tau > 0$. In fact, we must prove that the set $\Omega$ of all $t$ such that $v(t) = 0$ is an open set. But since such a set is nonempty, for any $t_1 \in \Omega$, proving that $[t_1, t_1 + \tau) \subset \Omega$ follows a similar argument.

In the rest of the proof, we will work entirely on the slab $[t_0, t_0 + \tau) \times \mathbb{R}^N$. Since $v(t_0) = 0$, we can take $\tau$ as small as necessary such that
\[
\|\nabla v \|_{W(I)} \lesssim \|\nabla v\|_{S(L^2, I)} \leq \eta,
\]
where $\eta$ is a small absolute constant to be chosen later. Similarly, from the hypothesis $\nabla u_2 \in S(L^2,I)$ and satisfying \eqref{lwp1}, we have
\[
\|\nabla u_2\|_{W(I)} \lesssim \eta.
\]

Furthermore, since $u_j$, $j = 1, 2$, satisfy \eqref{duhamell}, we have
\[
v(t) = -i\int_{t_0}^{t} e^{i(t - s)\Delta}\left[ F(u_1(s)) - F(u_2(s))\right] \, ds.
\]

Hence,
\begin{equation*}
	\begin{split}
		\|v\|_{V(I)} \lesssim & \, \|F(u_1) - F(u_2)\|_{2, \frac{2N}{N+2}} \\
		\lesssim & \left(\|\nabla u_1\|_{W(I)}^p +  \|\nabla u_2\|_{W(I) }^p  \right)\|v\|_{V(I)} \\
		\lesssim & \, \left(\|\nabla v\|_{W(I)}^p +  \|\nabla u_2\|_{W(I) }^p  \right)\|v\|_{V(I)} \\
		\lesssim & \, \eta^{p} \|v\|_{V(I)}.
	\end{split}
\end{equation*}

As $v$ was already finite in $V$, we conclude (by taking $\eta$ small enough) that $v$ vanishes almost everywhere in $[t_0, t_0 + \tau)$.

\vspace{0.5cm}
\textbf{Continuous Dependence in Rough Norms.} We now turn to the Lipschitz bound. Let $u$ and $\tilde{u}$ be solutions of \eqref{IVP} with initial data $u_0$ and $\tilde{u}_0$, respectively. From Strichartz estimates, Remark \ref{remarkcont}, and \eqref{propS}, we obtain
\begin{equation*}
	\begin{split}
		\|u - \tilde{u}\|_{S(L^2, I)} \lesssim & \, \|u_0 - \tilde{u}_0\| + \|F(u) - F(\tilde{u})\|_{2, \frac{2N}{N+2}} \\
		\lesssim & \, \|u_0 - \tilde{u}_0\| + \left(\|\nabla u\|_{W(I)}^p +  \|\nabla \tilde{u}\|_{W(I) }^p  \right)\|v\|_{V(I)} \\
		\lesssim & \, \|u_0 - \tilde{u}_0\| + \left(\|\nabla u\|_{W(I)}^p +  \|\nabla \tilde{u}\|_{W(I) }^p  \right)\|v\|_{S(L^2, I)}.
	\end{split}
\end{equation*}
Since both $u$ and $\tilde{u}$ satisfy \eqref{lwp1}, we conclude that
$$
\|v\|_{S(L^2, I)} \lesssim \, \|u_0 - \tilde{u}_0\| + \eta^p\|v\|_{S(L^2, I)}.
$$
By choosing $\eta$ small enough, we obtain $\|v\|_{S(L^2, I)} \lesssim \, \|u_0 - \tilde{u}_0\|$ as desired.

 \end{proof}

An immediate consequence of the above theorem and the Strichartz estimates is the global well-posedness result below.

\begin{theorem}[Global Well-Posedness for Small $H^{1}_{x}$ Data]\label{gwpECnls}
	Let $w_0 \in H^{1}_{x}$ be such that
	\begin{equation}\label{datasmall}
		\|w_0\|_{\dot{H}^{1}_{x}} \leq \eta_0
	\end{equation}
	for some small absolute constant $\eta_0$ depending only on the dimension $N$. Then, there exists a unique global solution $w$ such that $\nabla u \in S(L^2, \Real)$ for \eqref{IVP}, satisfying the following estimates:
	\begin{equation}\label{estsol1}
		\| \nabla w\|_{W(\Real)}, \|\nabla w\|_{S(L^{2}, \mathbb{R})} \lesssim \|w_0\|_{\dot{H}^{1}_{x}}
	\end{equation}
	\begin{equation}\label{estsol2}
		\| w\|_{S(L^{2}, \mathbb{R})} \lesssim \|w_0\|_{L^{2}_{x}}.
	\end{equation}
\end{theorem}

Now, we are going to show that control of the solution in the specific norm $\| \cdot \|_{Z(I)}$ yields control of the solution in all the $S^1(L^2, I)$-norms.

\begin{lemma}\label{boundedINLScritical}
	Let $I$ be a compact time interval, and let $w$ be the unique solution to \eqref{IVP} on $I \times \mathbb{R}^{N}$ obeying the bound 
	$$
	\|w\|_{Z(I)} \leq L.
	$$
	Then, if $t_0 \in I$ and $w(t_0) \in {H}^{1}$, we have
	$$
	\| w\|_{S(L^2, I)} \leq C(L) \|w(t_0)\|_{L_x^2}
	$$
	and
	$$
	\|\nabla  w\|_{S(L^2, I)} \leq C(L) \|w(t_0)\|_{ \dot{{H}^{1}_x}}.
	$$
\end{lemma}

\begin{proof}
	Subdivide the interval $I$ into $N_0 = C(L)$ subintervals $I_j = [I_j , I_{j+1}]$ such that
	$$
	\|w\|_{Z(I_j)} \leq \eta,
	$$
	where $\eta$ is a small positive constant to be chosen later. By Strichartz and Corollary \ref{cor2}, on each $I_j$ we obtain
	\begin{equation*}
		\begin{split}
			\|\nabla  w\|_{S(L^2, I_j)} \lesssim & \, \|w(t_0)\|_{ \dot{{H}^{1}_x}} + \|\nabla (|x|^{-b}|w|^{p}w)\|_{2, \frac{2N}{N+2}} \\
			\lesssim & \, \|w(t_0)\|_{ \dot{{H}^{1}_x}} + \|w\|_{Z (I_j)}^{p}\|\nabla w\|_{W(I_j)} \\
			\lesssim & \, \|w(t_0)\|_{ \dot{{H}^{1}_x}} + \eta^{p}\, \|\nabla w\|_{S(L^2 , I_j)}.
		\end{split}
	\end{equation*}
	
	Choosing $\eta$ sufficiently small, we obtain
	$$
	\|\nabla  w\|_{S(L^2, I_j)} \lesssim  \|w(t_0)\|_{ \dot{{H}^{1}}} .
	$$
	
	Similarly, we obtain
	$$
	\|w\|_{S(L^2, I_j)} \lesssim  \|w(t_0)\|_{ L^2_x} + \eta^p \|w\|_{S(L^2, I_j)}.
	$$
	
	The conclusion follows by adding these estimates over all subintervals $I_j$.
\end{proof}

 \subsection{Short-time pertubations}
 Next, we aim to establish a stability result for the \(H^1_x\)-critical INLS equation. This result pertains to the following property: Given an approximate solution
 \[
 i \tilde{u}_t + \Delta \tilde{u} = |x|^{-b}|\tilde{u}|^{\frac{4-2b}{N-2}} \tilde{u} + e,
 \]
 with \(\tilde{u}(0, x) = \tilde{u}_0(x) \in H^1(\mathbb{R}^n)\), where \(e\) is small in a suitable space and \(\tilde{u}_0 - u_0\) is small in \(H^1_x\), there exists a genuine solution \(u\) to \eqref{IVP} that remains very close to \(\tilde{u}\) in \(H^1_x\)-critical norms. 


  \begin{theorem}\label{STP}
  	Let $I$ be a compact time interval, $2 <  N < 6$, $p= \frac{4-2b}{N-2}$, and $0 < b < \min \{2, \frac{N}{2},\frac{ 6-N}{2} \}$. Let $\tilde{u}$ be an approximate solution to \eqref{IVP} on $I \times \mathbb{R}^{N}$ in the sense that 
  	$$
  	i\partial_t \tilde{u} + \Delta \tilde{u} =  |x|^{-b}|\tilde{u}|^{\frac{4-2b}{N-2}}\tilde{u} + e  
  	$$
  	for some function $e$. Suppose that we also have the energy bound 
  	\begin{equation}\label{energybound1}
  		\|\tilde{u}\|_{L_{t}^{\infty}\dot{H}_{x}^{1}(I \times \mathbb{R}^{N})} \leq E 
  	\end{equation}
  	for some constant $E >0$. Let $t_0 \in I$, and let $u(t_0) \in H^{1}_x$ be close to $\tilde{u}(t_0)$ in the sense that
  	\begin{equation}\label{energybound2}
  		\|u(t_0) - \tilde{u}(t_0)\|_{\dot{H}^{1}_x} \leq E'
  	\end{equation}
  	for some $E'>0$. Moreover, assume the smallness conditions
  	\begin{equation}\label{smallhip1}
  		\|\nabla \tilde{u}\|_{W(I)}  \leq \varepsilon_0
  	\end{equation}
  	\begin{equation}\label{smallhip2}
  		\|\nabla [  e^{i(t-t_0)\Delta}(u(t_0) - \tilde{u}(t_0))]\|_{W(I)} \leq \varepsilon
  	\end{equation}
  	\begin{equation}\label{smallhip3}
  		\|\nabla  e\|_{2, \frac{2N}{N+2}}  \leq \varepsilon
  	\end{equation} 
  	for some $0 < \varepsilon \leq \varepsilon_0$, where $\varepsilon_0 = \varepsilon_0(E, E')>0$ is a small constant. Then there exists a solution $u \in S^{1}(I \times \mathbb{R}^{N})$ to \eqref{IVP} on $I \times \mathbb{R}^{N}$ with the specified initial data $u(t_0)$ at time $t=t_0$ that satisfies
  	\begin{equation}\label{smallthesis1}
  		\|\nabla (u - \tilde{u})\|_{W(I)}  \lesssim \varepsilon 
  	\end{equation}
  	\begin{equation}\label{smallthesis2}
  		\|\nabla (u - \tilde{u})\|_{S(L^{2},I)}  \lesssim E' + \varepsilon
  	\end{equation} 
  	\begin{equation}\label{smallthesis3}
  		\| \nabla u \|_{S(L^{2},I)}  \lesssim E + E'
  	\end{equation} 
  	\begin{equation}\label{smallthesis4}
  		\left\| \nabla F(u, \tilde{u}) \right\|_{2, \frac{2N}{N+2}}  \lesssim  \varepsilon ,
  	\end{equation} 
  	where $F(u, \tilde{u}):= |x|^{-b}\left[|u|^{\frac{4-2b}{N-2}}u -  |\tilde{u}|^{\frac{4-2b}{N-2}}\tilde{u} \right]$.
  \end{theorem}

  \begin{proof} 
We can consider that there exist a solution $u$ for \eqref{IVP} given by Theorem \ref{lwp}. Assume that $t_0 = \inf I$. Let $v:= u - \tilde{u}$. Then $v$ satisfies the following IVP 
\begin{equation}\label{IVP2'}
 \begin{cases}
 i \partial_t v + \Delta v = F(v+ \tilde{u}, \tilde{u}) - e  \ \ \ t \in \mathbb{R}, \ x \in \mathbb{R}^{n}\\
 v(x,t_0) = v_0(x) = u(x, t_0) - \tilde{u}(x, t_0)  \in H^{1}(\mathbb{R}^{n}).
 \end{cases}
 \end{equation}
 
Now, we will work on the slab $[t_0, T] \times \mathbb{R}^{n} = I_1 \times \mathbb{R}^{n} $, for $T \in I$. By Strichartz estimates,
 \begin{equation*}
\|\nabla v\|_{W(I)} \lesssim 
 \| \nabla [e^{i(t-t_0)\Delta}v(t_0)]\|_{W(I)} + \|\nabla[F(v+ \tilde{u}, \tilde{u})]\|_{\gamma', \rho'} + \|\nabla e\|_{\gamma' , \rho'}
 \end{equation*}
for any $L^{2}$-admissible pair $(\gamma, \rho)$. So, from \eqref{smallhip2} and \eqref{smallhip3} results
\begin{equation}\label{bounded1}
 \|  \nabla v\|_{  W(I) } \lesssim  \|\nabla [ e^{i(t-t_0)\Delta}v(t_0)]\|_{W(I)} + \|\nabla[F(v+ \tilde{u}, \tilde{u})]\|_{2, \frac{2N}{N+2}}  +  \varepsilon  \lesssim  \|\nabla[F(v+ \tilde{u}, \tilde{u})]\|_{2, \frac{2N}{N+2}}  +  \varepsilon .
\end{equation} 
In turn, since $0< b < \min \{2, \frac{N}{2}, \frac{6-N}{2} \}$ and $2 < N < 6$ we have $p=\frac{4-2b}{N-2}>1$. Hence, 
$$
\big| \nabla \left[ F(v+\tilde{u}, \tilde{u}) \right]  \big| \lesssim \left[ |x|^{-(b+1)}\left( |v+\tilde{u}|^{p} +|\tilde{u} |^{p} \right) \right]|v| + |x|^{-b} |v+\tilde{u}|^{p}|\nabla v| + |x|^{-b} \left(  |v+\tilde{u}|^{p -1} +|\tilde{u} |^{p -1}\right)|\nabla \tilde{u}||v|.
$$

So, it follows from \eqref{nonlinearestimate3} that
\begin{equation*}
\begin{split}
S(t) := & \; \|\nabla[F(v+ \tilde{u}, \tilde{u})]\|_{2, \frac{2N}{N+2}} \leq  \;  \||x|^{-(b+1)}  |v+ \tilde{u}|^{p}v  \|_{2, \frac{2N}{N+2}} + \||x|^{-(b+1)}  |\tilde{u}|^{p}v  \|_{2, \frac{2N}{N+2}}  + \||x|^{-b} |v+ \tilde{u}|^{p} \nabla v  \|_{2, \frac{2N}{N+2}}  \\
 & \;  + \||x|^{-b}|v+ \tilde{u}|^{p-1}|v| \nabla \tilde{u}\|_{2, \frac{2N}{N+2}}+  \||x|^{-b}| \tilde{u}|^{p-1}|v| \nabla \tilde{u}\|_{2, \frac{2N}{N+2}}\\
  \lesssim & \; \|\nabla [v + \tilde{u}]\|_{W(I)}^{p}\| \nabla v \|_{W(I)} + \|\nabla  \tilde{u} \|_{W(I)}^{p} \| \nabla v \|_{W(I)}  +  \|\nabla (v + \tilde{u})\|_{W(I)}^{p -1} \|\nabla   \tilde{u} \|_{W(I)} \| \nabla v \|_{W(I)} .
 \end{split}
\end{equation*}

The previous inequality, \eqref{bounded1} and \eqref{smallhip1} imply that
\begin{small}
\begin{equation*}
\begin{split}
S(t) \lesssim & \; \| \nabla  v\|_{ W(I)}^{p +1}  + \| \nabla v\|_{  W(I) }^{p } \| \nabla \tilde{u}\|_{ W(I) } + \| \nabla  \tilde{u}\|_{  W(I)}^{p} \| \nabla v\|_{  W(I) }\\
\lesssim \; &  \left(S(t) + \varepsilon \right)^{p +1} + \varepsilon_0 \left( S(t) + \varepsilon \right)^{p } + \varepsilon_0^{p} \left( S(t) +  \varepsilon \right).
\end{split}
\end{equation*}
\end{small} 

Thus, a standard continuity argument shows that if we take $\varepsilon_0 = \varepsilon_0 (E, E')$ sufficiently small we obtain 
\begin{equation}\label{bounded2}
\|\nabla[F(v+ \tilde{u}, \tilde{u})]\|_{2, \frac{2N}{N+2}} \leq \varepsilon \ \ \mbox{for all } \ T \in I_1,
\end{equation}
which implies \eqref{smallthesis4}. Using \eqref{bounded1} and \eqref{bounded2}, one easily derives \eqref{smallthesis1}.  

Also, by Lemma \ref{strichartzest} we obtain
\begin{equation*}
\begin{split}
\| \nabla (u - \tilde{u})\|_{S(L^{2}, I)}  \lesssim  &\|u(t_0) - \tilde{u}(t_0)\|_{\dot{H}^{1}_x} + \|\nabla[F(u, \tilde{u})]\|_{2, \frac{2N}{N+2}} + \|\nabla e\|_{2, \frac{2N}{N+2}} 
\lesssim  E' + \varepsilon,
\end{split}
\end{equation*}
where in the last inequality we use \eqref{energybound2} and \eqref{smallhip3}.

Now, an application of \eqref{energybound1}-\eqref{energybound2}, \eqref{nonlinearestimate4} and \eqref{smallthesis2} results 
\begin{equation*}
\begin{split}
\|\nabla u\|_{S(L^{2}, I)} \lesssim & \; \|u(t_0)\|_{\dot{H}^{1}_x} + \|\nabla[|x|^{-b}|u|^{p}u]\|_{2, \frac{2N}{N+2}} \lesssim  \|u(t_0) - \tilde{u}(t_0)\|_{\dot{H}^{1}_x} +  \|\tilde{u}(t_0)\|_{\dot{H}^{1}_x} + \| \nabla u\|_{W(I)}^{p +1} \\
\lesssim & \; E + E' + \left(  \|  \nabla v \|_{W(I)} +   \| \nabla  \tilde{u} \|_{W(I)} \right)^{p +1} \lesssim E + E' +( \varepsilon + \varepsilon_0)^{p +1},
\end{split}
\end{equation*}
which proves \eqref{smallthesis3}, provide $\varepsilon_0$ is sufficiently small depending on $E$ and $E'$. 
  \end{proof}
  
  \subsection{Long-Time Perturbations}
  
  Using a straightforward iterative argument (similar to \cite{CKSTT2008}, \cite{RV2007}, and \cite{TV2005}) based on time interval partitioning, we can relax hypothesis \eqref{smallhip1} by allowing $\tilde u$ to be large (while still remaining bounded in some norm).
  
  \begin{theorem}\label{LTP}
  	Let $I$ be a compact time interval, $2<  N < 6$, $p= \frac{4-2b}{N-2}$, and $0 < b < \min \{2, \frac{N}{2},\frac{ 6-N}{2} \}$. Suppose $\tilde{u}$ is an approximate solution to \eqref{IVP} on $I \times \mathbb{R}^{N}$ in the sense that 
  	$$
  	i\partial_t \tilde{u} + \Delta \tilde{u} =  |x|^{-b}|\tilde{u}|^{\frac{4-2b}{N-2}}\tilde{u} + e  
  	$$
  	for some function $e$. Assume we have the energy bound 
  	\begin{equation}\label{energybound1-L}
  		\|\tilde{u}\|_{L_{t}^{\infty}\dot{H}_{x}^{1}(I \times \mathbb{R}^{N})} \leq E 
  	\end{equation}
  	for some constant $E >0$ and
  	\begin{equation}\label{boundedhipL}
  		\|\tilde{u}\|_{Z(I)} \leq L
  	\end{equation} 
  	for some $L>0$. Let $t_0 \in I$, and let $u(t_0) \in H^{1}(\mathbb{R}^{n})$ be close to $\tilde{u}(t_0)$ in the sense that
  	\begin{equation}\label{energybound2-L}
  		\|u(t_0) - \tilde{u}(t_0)\|_{\dot{H}^{1}_x} \leq E'
  	\end{equation}
  	for some $E'>0$. Moreover, assume the smallness conditions:
  	\begin{equation}\label{smallhip2-L}
  		\|\nabla [ e^{i(t-t_0)\Delta}(u(t_0) - \tilde{u}(t_0))]\|_{W(I)} \leq \varepsilon
  	\end{equation}
  	\begin{equation}\label{smallhip3-L}
  		\|\nabla  e\|_{2, \frac{2N}{N+2}}  \leq \varepsilon
  	\end{equation} 
  	for some $0 < \varepsilon \leq \varepsilon_1$, where $\varepsilon_1 = \varepsilon_1(E, E',M)>0$ is a small constant. Then there exists a solution $u \in S^{1}(I \times \mathbb{R}^{N})$ to \eqref{IVP} on $I \times \mathbb{R}^{N}$ with the specified initial data $u(t_0)$ at time $t=t_0$ that satisfies:
  	\begin{equation}\label{smallthesis1-l}
  		\|\nabla (u - \tilde{u})\|_{W(I)}  \lesssim C(E, E', L) \varepsilon 
  	\end{equation}
  	\begin{equation}\label{smallthesis2-L}
  		\|\nabla (u - \tilde{u})\|_{S(L^{2},I)}  \leq C(E, E', L) \left( E' + \varepsilon  \right)
  	\end{equation} 
  	\begin{equation}\label{smallthesis3-L}
  		\| \nabla u \|_{S(L^{2},I)}  \leq C( E + E',L).
  	\end{equation} 
  	Here, $C(E, E', L)>0$ is a non-decreasing function of $E$, $E'$, $L$, and the dimension $N$.
  \end{theorem}

\begin{proof}
We will derive Theorem \ref{LTP} from Theorem \ref{STP} using an iterative procedure. First, we assume without loss of generality that $t_0 = \inf I$. Let $\varepsilon_0 = \varepsilon_0(E, 2E')$ be as defined in Theorem \ref{STP}.

The first step is to establish a bound on $\|\nabla \tilde{u}\|_{S(L^{2}, I)}$. To do so, we subdivide $I$ into $N_0 \sim \left( 1 + \frac{L}{\varepsilon_0} \right)^{\frac{2(N+2-2b)}{N- 2}}$ subintervals $J_k$ such that
\begin{equation}\label{ineq10}
	\| \tilde{u}\|_{Z(J_k)} \leq \varepsilon_0.
\end{equation}

By combining Corollary \ref{cor2} with Lemma \ref{strichartzest}, \eqref{energybound1-L}, \eqref{smallhip3-L}, and \eqref{ineq10}, we estimate
\begin{equation*}
	\begin{split}
		\|\nabla \tilde{u}\|_{S(L^{2}, J_k)} \lesssim &\;  \|\tilde{u}(t_0)\|_{H^{1}_x} + \|\nabla(|x|^{-b}|\tilde{u}|^{p} \tilde{u})\|_{2, \frac{2N}{N+2}} + \|\nabla e\|_{2, \frac{2N}{N+2}}\\
		\lesssim & \; E + \| \tilde{u}\|_{Z(J_k)}^{p} \|\nabla \tilde{u}\|_{W(J_k)} + \varepsilon \\
		\lesssim & \; E + \varepsilon_0^{p} \| \nabla \tilde{u}\|_{S(L^{2}, J_k)} + \varepsilon.
	\end{split}
\end{equation*}

A standard continuity argument yields
$$
\| \nabla \tilde{u}\|_{S(L^{2}, J_k)} \lesssim E,
$$
provided $\varepsilon_0$ is sufficiently small depending on $E$. Summing these bounds over all the intervals $J_k$, we obtain
$$
\|\nabla \tilde{u}\|_{S(L^{2}, I)} \leq C(E, L,  \varepsilon_0),
$$
which implies
$$
\|\nabla \tilde{u}\|_{W( I)} \leq C(E, L,  \varepsilon_0).
$$

Next, we further subdivide $I$ into $N_1 =C(E, L, \varepsilon_0)$ subintervals $I_j = [t_j, t_{j+1}]$ such that
\begin{equation}\label{condaux1}
	\|\nabla \tilde{u}\|_{ W ( I_j ) } \leq \varepsilon_0.
\end{equation}

Choosing $\varepsilon_1$ sufficiently small depending on $N_1$, $E$, and $E'$, we apply Theorem \ref{STP} to obtain for each $j$ and all $0 < \varepsilon < \varepsilon_1$
\begin{equation*}
	\|\nabla (u - \tilde{u})\|_{W(I_j)}  \leq C(j) \;  \varepsilon 
\end{equation*} 
\begin{equation*}
	\|\nabla (u - \tilde{u})\|_{S(L^{2},I_j)}  \leq C(j) \left(E' +\varepsilon \right)
\end{equation*} 
\begin{equation*}
	\|\nabla u \|_{S(L^{2},I_j)}  \leq C(j) \left( E + E' \right)
\end{equation*} 
\begin{equation*}
	\left\| \nabla F(u, \tilde{u}) \right\|_{2, \frac{2N}{N+2}}  \leq C(j) \;  \varepsilon ,
\end{equation*} 
provided we can show that \eqref{energybound2-L} holds with $t_0$ replaced by $t_j$. This follows by an inductive argument. By Strichartz estimates, \eqref{energybound2-L}, and the inductive hypothesis
\begin{align}
	\|u(t_{j+1}) - \tilde{u}(t_j+1)\|_{\dot{H}^{1}_x} \lesssim & \; \|u(t_{0}) - \tilde{u}(t_0)\|_{\dot{H}^{1}_x} + \|\nabla e\|_{L^{2}_{ [t_0, t_{j+1}]}L^{\frac{2}{N+2}}_{x}} + \left\| \nabla F(u, \tilde{u}) \right\|_{L^{2}_{ [t_0, t_{j+1}]}L^{\frac{2}{N+2}}_{x}}\\
	\lesssim & E' + \varepsilon +\displaystyle \sum_{k=0}^{j} C(k) \;  \varepsilon 
\end{align}

By Strichartz estimates, we have
\begin{equation*}
	\begin{split}
		\|\nabla[ e^{i(t-t_{j+1})\Delta}(u(t_{j+1} ) - \tilde{u}(t_{j+1}))]\|_{W(I)} \lesssim & \|\nabla[ e^{i(t-t_0)\Delta}(u(t_0 ) - \tilde{u}(t_0 ))]\|_{W(I)} + \| \nabla e\|_{2, \frac{2N}{N+2}}+   \| [\nabla F(u, \tilde{u})] \|_{2, \frac{2N}{N+2}} \\
		\lesssim & \; \varepsilon + \|\nabla e\|_{2, \frac{2N}{N+2}} + \|\nabla F(u, \tilde{u})\|_{2, \frac{2N}{N+2}} \\
		\lesssim & \; \varepsilon + \displaystyle \sum_{k=0}^{j} C(k) \;   \varepsilon .
	\end{split}
\end{equation*}
Here, $C(k)$ depends only on $k$, $E$, $E'$, and $\varepsilon_0$. Choosing $\varepsilon_1$ sufficiently small depending on $N_1$, $E$, and $E'$, we can continue the inductive argument. Note that the final constants can easily be chosen to depend in a non-decreasing manner on $E$, $E'$, and $L$.

\end{proof}

 \section{Local well-posedness for DINLS}
 In this section, we prove the local well-posedness for the IVP \eqref{IVP2} with $p_2=\frac{4-2b_2}{N-2}$.
 
 \begin{proposition}[Local Well-Posedness]\label{lwp}
 	Let $u_0 \in H^{1}_x$, $\lambda_1, \lambda_2$ be nonzero real constants, $2<  N <6$, $0 < b_1, b_2 < \min \{\frac{N}{2}, 2 \}$,  $0 < p_1 < \frac{4-2b_1}{N-2}$, and $p_2 = \frac{4-2b_2}{N-2}$. Given $u_0 \in H^{1}(\mathbb{R}^{N})$, there exists $T=T(u_0)>0$ and a unique solution $u$ of \eqref{IVP2} in the time interval $[-T, T]$ with
 	$$
 	u \in C([-T,T]: H^1(\mathbb{R}^{N})) \cap X^1_{b_1, b_2}([-T,T]).
 	$$ 
 	In addition, for all $T'< T$, there exists a neighborhood $B$ of $u_0$ in $H^{1}(\mathbb{R}^{N})$ such that the function $u_0 \in B \mapsto u(t) \in C([-T,T]: H^1(\mathbb{R}^{N})) \cap X^1_{b_1, b_2}([-T,T])$ is Lipschitz.
 	Moreover, for every $T>0$, there exists $\eta = \eta(T)$ such that if
 	$$
 	\|e^{it \Delta } u_0 \|_{X^1_{b_1, b_2}([-T,T])}  \leq \eta,
 	$$
 	then \eqref{IVP2} admits a unique strong $H^{1}_x$-solution $u$ defined on  $[-T,T]$.  
 \end{proposition}

 \begin{proof}

Without loss of generality, we consider only the case $t>0$. Let $\Gamma$ be the operator given by
\begin{equation*}
	\Gamma(u)(t)  =  e^{it\Delta}u(t_0) - i \int_{0}^{t} e^{i(t-\tau)\Delta}G(\cdot, \tau)\, d\tau,
\end{equation*}
where $G(u) = \lambda_1|x|^{-b_1}|u|^{p_1}u + \lambda_2 |x|^{-b_2}|u|^{p_2}u$.

Note that from Strichartz estimates and Corollary  \ref{nlestimates2}, we have for $k=0,1$ we obtain
\begin{align}\label{lwp22}
	\nonumber \| |\nabla|^{k} \Gamma (u)(t)\|_{X_{b_1, b_2}([-T,T])} \lesssim  & \,   \||\nabla|^{k}e^{it\Delta} u_0\|_{X_{b_1, b_2}([-T,T])} + \||\nabla|^{k}G(u)\|_{S'(L^2, [-T,T])} \\
	\nonumber \leq&  \,  C \||\nabla|^{k} e^{it\Delta} u_0\|_{X_{b_1, b_2}([-T,T])} + C\|\nabla u\|_{W([-T,T])}^{p_2}\||\nabla|^{k}u\|_{V([-T,T])}\\
	& \, + C  T^{1- \frac{p_1+2}{\gamma_{b_1}}}\| \nabla u\|_{X_{b_1, b_2}([-T,T])}^{p_1}\||\nabla|^{k}u\|_{X_{b_1, b_2}([-T,T])} .
\end{align}

On the other hand,  we have
\begin{align}\label{lwp12}
	\nonumber \sup_{[0, T_0]} \| \Gamma (u)(t) - e^{it\Delta} u_0\|_{H^{1}}   \leq & \,  C \sum_{k=0}^{1}\| \int_{0}^{t} e^{i(t - \tau)\Delta}|\nabla|^{k}G(u)\|_{S(L^2,I)} \\
	\nonumber \leq & \, C \,  \| \nabla u\|_{W([-T,T])}^{p_2}\||\nabla|^{k}u\|_{V([-T,T])} \\
	& \, + C  T^{1- \frac{p_1+2}{\gamma_{b_1}}} \, \| \nabla u\|_{X_{b_1, b_2}([-T,T])}^{p_1}\||\nabla|^{k}u\|_{X_{b_1, b_2}([-T,T])} .
\end{align}

Since the $X_{b_1,b_2}$-norm is $L^2$-admissible, for any fixed $\epsilon >0$, there exists $\tilde{T}>0$ (see \cite{PL}, Proposition 5.1) such that
$$
\|e^{it\Delta} u_0\|_{ X^1_{b_1, b_2}([-\tilde{T},\tilde{T}])}  + \|\nabla u\|_{X_{b_1,b_2}([-\tilde{T},\tilde{T}])}< \epsilon .
$$
By considering the norm 
$$
\|v\|_{\Sigma} := \sup_{[0, T_0]} \| v - e^{it\Delta} u_0\|_{H^{1}} + \| v\|_{X^1_{b_1, b_2}([-T,T])},
$$
we define the space
\begin{align*}
	\Sigma_{a}^T := \left\{v \ :  \ v \in C([0,T], H^{1}(\mathbb{R}^{N})) \cap X^1_{b_1, b_2}([0,T])  \ \ \mbox{and} \ \ \|u\|_{\Sigma}\leq a \right\}
\end{align*}
and the metric 
$$d(u,v) := \|u-v\|_{X_{b_1, b_2}([0,T])}.$$
By a similar argument as in Step 1 of the proof of [\cite{cazenave03} Theorem 4.4.1] (see also [\cite{cazenave03}, Theorem 1.2.5]), we can prove that $(\Sigma_M , d)$ is a complete metric space.

Using \eqref{lwp22} and \eqref{lwp12}, we have that for all $\epsilon >0$, there exists $T>0$ such that if $u \in \Sigma_a^T$, then
\begin{equation}\label{PF3}
	\|\Gamma (u)\|_{\Sigma} < C \epsilon + C \left( \epsilon^{p_2} + C T^{1- \frac{p_1+2}{\gamma_{b_1}}} a^{p_1} \right) a.
\end{equation}

Therefore, if
\begin{equation}\label{PF1}
	C \epsilon + C \left( \epsilon^{p_2} + CT^{1 - \frac{p_1+2}{\gamma_{b_1}}}a^{p_1}\right) a < a,
\end{equation}
we have that $\Gamma(\Sigma_a^T) \subseteq \Sigma_a^T$. 

Additionally, for the same $T(\epsilon)$ from \eqref{PF3}, we obtain
\begin{equation*}
	\begin{split}
		\| \Gamma (u) - \Gamma (v) \|_{ X_{b_1, b_2}}  \lesssim & \,  \| u -v \|_{S(L^2, [-T,T])} \lesssim    \| G(u) - G(v)\|_{\gamma', \rho'} \\
		\lesssim &  \,  \displaystyle 2 C\left(\|\nabla u\|_{W([-T,T])}^{p_2} + T^{1- \frac{p_1+2}{\gamma_{b_1}}}\|\nabla  v\|_{X_{b_1, b_2}([-T,T])}^{p_1} \right)\|u-v\|_{X_{b_1, b_2}([-T,T])}\\
		\leq &\,   2C\left(\epsilon^{p_2} + T^{1- \frac{p_1+2}{\gamma_{b_1}}}a^{p_1 } \right) \|u-v\|_{X_{b_1, b_2}([-T,T])}.
	\end{split}
\end{equation*}

Thus, for
\begin{equation}\label{PF2}
	2C\left( \epsilon^{p_2} + T^{1- \frac{p_1+2}{\gamma_{b_1}}}a^{p_1 } \right)  < \frac{1}{2},
\end{equation}
we have that $\Gamma$ is a contraction. Now, fixing $\epsilon $ sufficiently small and choosing $a$ satisfying \eqref{PF2}, where $T$ is given by $\epsilon$, we have that \eqref{PF1} is verified. 

Therefore, by the Point Fix Theorem, we establish the existence and uniqueness of the solution to \eqref{IVP2}. To prove the continuous dependence of $\Gamma(u)$ with respect to $u_0$, note that if $u$ and $v$ are the corresponding solutions of \eqref{IVP2} with initial data $u_0$ and $v_0$, respectively, and $\|u_0 - v_0\|$ is sufficiently small, the same argument used in \eqref{lwp22} and \eqref{lwp12} implies
\[
\|u-v\|_{\Sigma} \leq K\|u_0 - v_0\|_{H^1},
\]
which shows that the function $u_0 \mapsto u(t)$ from a neighborhood $B$ of $u_0$ in $H^{1}(\mathbb{R}^{N})$ on $C([-T,T]: H^1(\mathbb{R}^{N})) \cap X^1_{b_1, b_2}([-T,T])$ is Lipschitz.

In turn, given $T > 0$, if we choose $a = 2C \|e^{it\Delta}u_0\|_{X^1_{b_1, b_2}([-T,T])}$ and take $\eta = \eta(T)$ satisfying $a \leq 2C\eta(T)$ and $2C \left( T^{1- \frac{p_1+2}{\gamma_{b_1}}}a^{p_1 } + a^{p_2}\right) \leq \frac{1}{4}$, the inequalities \eqref{PF2} and \eqref{PF1} hold with $\eta$ replaced by $\epsilon$. By a similar argument as before, the second part of the theorem has been proven.
\end{proof}

\begin{proposition}[Blow-up alternative]
	Let $(-T_{\min}, T_{\max})$ be the maximal time interval on which the unique strong $H^{1}_x$-solution $u$ of \eqref{IVP2} is well-defined. Then, $u \in S^{1}(I \times \mathbb{R}^{N})$ for every compact time interval $I \subset (-T_{\min}, T_{\max})$, and the following properties hold:
	\begin{itemize}
		\item[i)] If $T_{\max} < \infty$, then 
		\[
		\lim_{t \rightarrow T_{\max}} \|\nabla u(t)\| = \infty \quad \text{or} \quad \|\nabla u\|_{S(L^{2}, (0, T_{\max}))} = \infty.
		\]
		\item[ii)] If $T_{\min} < \infty$, then 
		\[
		\lim_{t \rightarrow T_{\min}} \|\nabla u(t)\| = \infty\quad \text{or} \quad \|\nabla u\|_{S(L^{2}, (T_{\min}, 0))} = \infty.
		\]
	\end{itemize}
\end{proposition}
\begin{proof}
We will prove item (i). The second item is analogous.

Since the solution $u$ of \eqref{IVP2} belongs to $C([0,T_{\max}]: H^1(\mathbb{R}^{N})) \cap X^1_{b_1, b_2}([0,T_{\max}])$ and $T_{\max}$ satisfies \eqref{PF2}, we have
\[
\|u\|_{S^1([0, T_{\max}] \times \mathbb{R}^{N})} \leq C\|u_0\|_{H^1_x} + C\|u\|_{X^1_{b_1, b_2}([0, T_{\max}])} \left(\|\nabla u\|_{X_{b_1, b_2}([0, T_{\max}])}^{p_2} + T_{\max}^{1- \frac{p_1+2}{\gamma_{b_1}}}\|\nabla u\|_{X_{b_1, b_2}([0, T_{\max}])}^{p_1}\right) < \infty.
\]

Suppose that there exists a sequence $\{t_j\}$ and a constant $M$ such that $t_j \rightarrow T_{\max}$ and 
$
\|\nabla u(t_j)\| +\|\nabla u\|_{X_{b_1,b_2}} 
\leq M
$
for all $j \in \mathbb{N}$. Note that if $\|\nabla u(t_j)\| \leq M$, from Strichartz estimates  $\|e^{i(t-t_j)\Delta}\nabla u \|_{X_{b_1, b_2}} \lesssim \|\nabla u(t_j)\|$. So, there exists a $\tilde{T}$ sufficently small such that $\|e^{i(t-t_j)\Delta}\nabla u_0\|_{X_{b_1, b_2}([-\tilde{T}, \tilde{T}])}$ is small. Let $T(M)$ denote the existence time obtained by Theorem \ref{lwp} for all initial data bounded above by $M$, and let $k$ be such that $t_k + T(M) > T_{\max}$.

There exists a unique strong $H^{1}_x$-solution $v_k$ of \eqref{IVP2} on $[t_k, T(M)]$ starting from $u(t_k)$. Gluing together $u$ with $v_k$, we obtain a solution of \eqref{IVP2} defined on $[0, t_k + T(M)]$, which is impossible since \(T_{\max}\) is the maximal time of existence..
\end{proof}

As a complement to the previous proposition, we have the following important lemma.

\begin{proposition}\label{Blow-upcriterion}
	Let $u_0 \in H^{1}_x$ and let $u$ be the unique strong $H^{1}_x$-solution of \eqref{IVP2} on the slab $[0, T_0] \times \mathbb{R}^{N}$ such that
	$$
	\|u\|_{X_{b_1,b_2}([0,T_0])} < \infty.
	$$
	Then, there exists $\delta = \delta(u_0) > 0$ such that the solution $u$ extends to a strong $H^{1}_x$-solution on the slab $[0, T_0 + \delta] \times \mathbb{R}^{N}$.
\end{proposition}

\section{Global well-posedness to INLS with two power-type nonlinearities}

In this section, we will prove Theorem \ref{gwp}. Our main tools will be the uniform control of the kinetic energy for solutions to (DINLS) equation and a``good'' local well-posedness theory for the IVP \eqref{IVP2} under some condictions.

\subsection{Kinetic Energy Control}

Now, let's estimate the $\dot{H}^{1}_x$-norm of the solution $u$ to \eqref{IVP2}. From the energy conservation \eqref{energy} for solutions of \eqref{IVP2}, we have
\begin{equation}\label{gradientbounded}
	\frac{1}{2} \|\nabla u (t)\|_{L_{x}^{2}}^{2} =  E(u_0) - \frac{\lambda_1}{p_1 +2}\int_{\mathbb{R}^{N}}  |x|^{-b_1}|u(t,x)|^{p_1+2} \; dx -   \frac{\lambda_2}{p_2 +2}\int_{\mathbb{R}^{N}}|x|^{-b_2}|u(t,x)|^{p_2 +2}  \; dx.
\end{equation}

Note that if $\lambda_1$ and $\lambda_2$ are both positive, it follows immediately that we have uniform control of the kinetic energy of the solution:
\[
\|\nabla u (t)\|_{L_{x}^{2}}^{2} \leq 2E(u_0), \quad \text{for all } t \in (-T_{\min}, T_{\max}).
\]

However, if $\lambda_1 < 0$ and $\lambda_2 > 0$, we need to control the potencial intercritical  energy in terms of the energy-critical term . For this we use the following proposition.

\begin{proposition}\label{bweightednorm}
Let $\eta >0$, $p_1 < \frac{4-2b_1}{N-2}$, $p_2 = \frac{4- b_2}{N-2}$, $b_1, b_2 < \min \{2, \frac{N}{2} \}$ and $u \in L^{2}(\mathbb{R}^{N}) \cap L^{p_2+2}_{b_2}(\mathbb{R}^{N})$. If 
\begin{itemize}
    \item[i)] $p_1 < p_2 $ and $\frac{p_1 }{p_2}b_2 \leq b_1 \leq b_2   $ or \\
    \item[ii)] $ b_2 < b_1 < \frac{N(p_2-p_1) }{p_2+2} $,
\end{itemize}

then $u \in L^{p_1+2}_{b_1}(\mathbb{R}^{N})$ and 
\begin{equation}\label{weakinterpolation}
\|u\|_{ L^{p_1+2}_{b_1}(\mathbb{R}^{N})}^{p_1+2} \leq \, \eta \, \|u\|_{ L^{p_2+2}_{b+2}(\mathbb{R}^{N})}^{p_2+2} + C \left( \eta \ , \ \|u\|_{ L^{2}(\mathbb{R}^{N})}\right).
\end{equation}

\end{proposition}

\begin{proof}

Let us remember the elementary inequality 
\begin{equation}\label{holderestimate}
AB \leq \left(\frac{A^{s}}{s} + \frac{B^{s'}}{s'} \right)
\end{equation}
valid for any $A$, $B>0$, with $1 < s < \infty$  and $s'$ the dual exponent to $s$. 
We split the $L^{p_1+2}_{b_1}(\mathbb{R}^{N})$-norm in two pieces, we must be estimate
$$
\|u\|_{ L^{p_1+2}_{b_1}(\mathbb{R}^{N})}^{p_1+2} = \displaystyle \int_{|x|\leq 1} \; |x|^{-b_1} |u|^{p_1 +1}|u| \, dx +  \displaystyle \int_{|x|> 1} \, |x|^{-b_1} |u|^{p_1}|u|^{2} \,  dx =: I + II.
$$
\textbf{(i)} $p_1 < p_2$ and $\frac{p_1 }{p_2}b_2 \leq b_1 \leq b_2$.\\
In the first integral, since $b_1 \leq b_2$ we can replace $-b_1$ by $-b_2 +t$ for some $t> 0$ and to apply \eqref{holderestimate} for $A=\eta^{\frac{1}{s}}|u|^{p_1+1}$, $B=\eta^{-\frac{1}{s}}$ and $s= \frac{2p_2+1}{2p_1+1} >1$, to obtain
\begin{equation*}
\begin{split}
I \leq & \displaystyle \int_{|x| \leq 1} \; |x|^{-b_2} \left[ \eta|u|^{p_2+1} + \eta^{-\frac{s'}{s}} \right]|u| \, dx  \leq  \eta \displaystyle \int_{|x| \leq 1} \; |x|^{-b_2} |u|^{p_2+2} \; dx + C_\eta \displaystyle \int_{|x| \leq 1} \; |x|^{-b_2}|u|\, dx\\
\leq & \; \eta\,  \|u\|_{ L^{p+2+2}_{b_2}(\mathbb{R}^{N})}^{p_2+2} + C_\eta \, \||x|^{-b_2 } \|_{L^{2}(\{ |x| \leq 1 \})} \, \|u\|_{L^{2}(\mathbb{R}^{N})} \leq  \eta \,  \|u\|_{ L^{p_2+2}_{b_2}(\mathbb{R}^{N})}^{p_2+2} + C_\eta  \, \|u\|_{L^{2}(\mathbb{R}^{N})},
\end{split}
\end{equation*}
where in the last inequality we use the fact the $b_2 < \frac{N}{2}$. 

In turn, we rewrite the integral $II$ by
$$
II = \displaystyle \int_{|x| >1 } \; |x|^{-b_1 + \frac{p_1b_2}{p_2}} \left[ |x|^{-\frac{p_1b_2}{p_2}}|u|^{p_1}\right] |u|^{2} \; dx \leq  \displaystyle \int_{|x| >1 } \; \left[ |x|^{-\frac{p_1b_2}{p_2}}|u|^{p_1}\right] |u|^{2} \; dx,
$$
where the last inequality is derived from  $\frac{p_1b_2}{p_2}\leq b_1$. Now, taking $A=  \eta^{\frac{1}{s}}|x|^{-\frac{p_1b_2}{p_2}}|u|^{p_1}$, $B=\eta^{-\frac{1}{s}}$ and $s= \frac{p_2}{p_1} >1$ in \eqref{holderestimate} we obtain
$$
II \leq  \eta \, \|u\|_{ L^{p_2+2}_{b_2}(\mathbb{R}^{N})}^{p_2+2} +  C_\eta \, \|u\|_{L^{2}(\mathbb{R}^{N})}^{2}
$$
and the proposition is showed for this case.

\textbf{(ii)} $b_2 < b_1 <  \frac{N(p_2-p_1)}{p_2+2}$.\\

In this case we have
\begin{equation*}
\begin{split}
    I\leq  & \left(\int_{|x|< 1} |x|^{-b_1(\frac{N}{b_1+\varepsilon})} \, dx \right)^{\frac{b_1 + \varepsilon }{N}} \left( \int_{ |x|< 1 } |u|^{\frac{N(p_1+2)}{N- b_1- \varepsilon}}\, dx \right)^{\frac{N - b_1 - \varepsilon }{N}}  
\leq C \left( \int_{ |x|< 1 } |u|^{\frac{N(p_1+2)}{N- b_1- \varepsilon}} \right)^{\frac{N - b_1 - \varepsilon }{N}},
\end{split}
\end{equation*}
where $\varepsilon$ chosen sufficiently small satisfying
\begin{equation}\label{cond 1epsilon}
\varepsilon < N \left[ 1 - \frac{b_1}{N} - \frac{p_1+2}{p_2+2}\right].
\end{equation}

By interpolation, since that $b_2 < \min\{2, \frac{N}{2}  \}$ and $p_1+2 < (p_2+2) \left(1 -  \frac{b_1}{N}\right)$ we obtain from last inequality that
\begin{equation*}
    \begin{split}
        I \leq C \|u\|^{\theta(p_1+2)} \left( \int_{ |x|< 1 } |u|^{p_2+2} \right)^{\frac{(p_1+2)}{p_2+2}(1-\theta)}  \leq C \|u\|^{\theta(p_1+2)}  \|u\|_{L^{p_2+2}_{b_2}}^{(p_1+2)(1-\theta)} .
    \end{split}
\end{equation*}
where 
\begin{equation}\label{deftheta}
\theta = \frac{2}{p_2} \left[ \frac{p_2+2}{p_1+2}\left( 1-   \frac{b_1}{N} \right) -1 - \varepsilon \frac{p_2+2}{N(p_1+2)} \right]
\end{equation}
Thus, using inequality \eqref{holderestimate} with $A = \eta^{\frac{1}{s}}\|u\|_{L^{p_2+2}_{b_2}}^{(p_1+2)(1-\theta)}$, $B=\eta^{-\frac{1}{s}}C \|u\|^{\theta(p_1+2)} $ and $s= \frac{p_2+2}{(p_1+2)(1- \theta) } >1$, we estimate
$$
I \leq \eta \|u\|_{L^{p_2+2}_{b_2}}^{(p_2+2)]} + C( \eta \, , \,  \|u\| )
$$
Now, we will estimate the integral II. For this, note that
\begin{equation}\label{interpI}
\begin{split}
\int_{|x|>1}|x|^{-b_1}|u|^{\frac{2(N-b_1)}{N-2 \varepsilon}} \leq & \int_{|x|>1}|x|^{-b_1+2 \varepsilon b_2\frac{(N-b_1)(N-2)}{(N-2 \varepsilon)(N- b_2)}}\left||x|^{-b_2\frac{(N-2)}{2(N- b_2)}}u\right|^{4\varepsilon\frac{(N-b_1)}{(N-2 \varepsilon)}}|u|^{\frac{2(N-b_1)}{N-2 \varepsilon}(1- 2 \varepsilon)}\,dx\\
\leq & \, C_N \|u\|_{L^2}^{\frac{4 \varepsilon (N-b_1)}{N-2 \varepsilon}}\left[\int|x|^{-b_2}|u|^{p_2+2}\,dx \right]^{(1- 2\varepsilon) \frac{2(N-b_1)}{(N-2 \varepsilon)(p_2+2)}},
\end{split}
\end{equation}
where
$$\frac{1}{N}\left(b_1 -(1-2 \varepsilon) b_2\frac{(N-b_1)(N-2)}{(N-2 \varepsilon)(N- b_2)}\right)>1- 2\varepsilon\frac{N- b_1}{N- 2\varepsilon } - (1-2\varepsilon) \frac{(N-b_1)(N-2)}{(N-2 \varepsilon)(N- b_2)} >0,$$
 what is always holds for $b_2 < \min \{2 , \frac{N}{2} \} $, $b_1 < N$ and $\varepsilon > 0 $ sufficiently small satisfying
 \begin{equation}\label{cond 2epsilon}
 \varepsilon < \frac{N(2-b_2)(b_1(N_2)}{2[(N-b_1)(2-b_2)+2(N-b_2)]}
 \end{equation}

 In the similar way, we obtain
\begin{equation}\label{interpII}
\begin{split}
\int_{|x|>1}|x|^{-b_1}|u|^{\frac{2(N-b_1)}{N-2+2 \varepsilon}} 
\leq & \, C_N \|u\|_{L^2}^{\frac{4\varepsilon(N-b_1)}{N-2+2 \varepsilon}   }\left[\int|x|^{-b_2}|u|^{p_2+2}\,dx \right]^{(1- 2 \varepsilon)\frac{2(N-b_1)}{(N-2+2 \varepsilon)(p_2+2)}},
\end{split}
\end{equation}
with
$$\frac{1}{N}\left(b_1 -(1-2 \varepsilon )b_2\frac{(N-b_1)(N-2)}{(N-2-2 \varepsilon)(N- b_2)}\right)>1- \frac{2\varepsilon (N- b_1)}{N-2+ 2\varepsilon } - (1-2 \varepsilon )\frac{(N-b_1)(N-2)}{(N-2+2 \varepsilon)(N- b_2)} >0. $$
We can reduce the condition above to $b_2 < b_1 < N$, $b_2 < \min \{ 2, \frac{N}{2}\}$ and 
\begin{equation}\label{cond 3epsilon}
   \begin{split}
0<\varepsilon<\frac{(N-2)(b_1-b_2)}{2[(N-b_1)(2-b_2)-(N-b_2)]}.  
   \end{split}
\end{equation}

By taking $\varepsilon $ satisfying \eqref{cond 2epsilon} and \eqref{cond 3epsilon}, and so interpolation \eqref{interpI} and \eqref{interpII} we obtain
\begin{equation}\label{interpIII}
\|u\|_{L_{b_1}^{p_{\beta}} (\{|x|>1\}) } \leq C \| u \|^{2 \varepsilon} \|u\|_{L_{b_2}^{p_2+2}  }^{1-2 \varepsilon} 
\end{equation}
where
$$
\frac{1}{p_{\beta}} = \beta \frac{N - 2 \varepsilon}{2( N -b_1)}+ (1-\beta) \frac{N -2+ 2 \varepsilon}{2( N -b_1)}.
$$
for any $\beta \in [0,1]$. In particular, we have \eqref{interpIII} holds for any $p_{\beta}$ satisfying $p_{\beta}+2 <  (p_2+2) \left(1- \frac{b_1}{N}\right) $. More precisely, for any $b_1 < \frac{N(p_2-p_1)}{p_2+2}.$ we have 
$$
II \leq C \| u \|^{(p_1+2)2\varepsilon} \|u\|_{L_{b_2}^{p_{2}+2}}^{(p_1+2 )(1-2\varepsilon) }  \leq  \eta \|u\|_{L^{p_2+2}_{b_2}}^{p_2+2} + C( \eta \, , \,  \|u\| ).
$$
\end{proof}

\begin{corollary}\label{benergycinetic}
	Consider the initial value problem (IVP) \eqref{IVP2} with \(0 < p_1 < \frac{4-2b_1}{N-2}\) and \(p_2 = \frac{4-2b_2}{N-2}\). Then, for all times \(t\) for which the solution is defined, we have
	\begin{equation}
		\|\nabla u(t)\|^2 \leq C(E(u_0), M(u_0))
	\end{equation}
	uniformly in time, under each of the following conditions
	\begin{itemize}
		\item[i)] \(\lambda_1, \lambda_2 > 0\);
		\item[ii)] \(\lambda_1 < 0\), \(\lambda_2 > 0\), \(p_1 < p_2\), and \(\frac{p_1b_2}{p_2} \leq b_1 \leq b_2 < \min\{2, \frac{N}{2}\}\);
		\item[iii)] \(\lambda_1 < 0\), \(\lambda_2 > 0\), and \(b_2 < b_1 < \frac{N(p_2-p_1)}{p_2+2}\).
	\end{itemize}
\end{corollary}

\begin{proof}
	We need to prove cases \(ii)\) and \(iii)\). From \eqref{gradientbounded}, we have
	\[
	\|\nabla u(t)\|^2 = 2E(u_0) + \frac{2|\lambda_1|}{p_1+2}\|u(t)\|_{L_{b_1}^{p_1 +2}}^{p_1+2} - \frac{2|\lambda_2|}{p_2+2}\|u(t)\|_{L_{b_2}^{p_2 +2}}^{p_2+2}.
	\]
	
	Applying Proposition \ref{bweightednorm} with \(\eta = \frac{(p_1 +2)|\lambda_2|}{(p_2+2)|\lambda_1|}\), we estimate
	\begin{equation*}
		\begin{split}
			\|\nabla u(t)\|^2 &= 2E(u_0) + \frac{2|\lambda_2|}{p_2+2}\|u(t)\|_{L_{b_2}^{p_2 +2}}^{p_2+2} \\
			&\quad - \frac{2|\lambda_2|}{p_2+2}\|u(t)\|_{L_{b_2}^{p_2 +2}}^{p_2+2} + C(M(u_0)) \\
			&\leq C(M(u_0), E(u_0)),
		\end{split}
	\end{equation*}
	where the constant does not depend on time.
\end{proof}

\subsection{Good Local Well-Posedness}

We know that the solution of the initial value problem (IVP) \eqref{IVP2}, as found in Proposition \ref{lwp}, belongs to \(C([-T,T] : H^1(\mathbb{R}^N)) \cap X^1_{b_1, b_2}([-T, T])\). Now, we aim to find a time \(T = T(\|u_0\|_{H^1_x})\) such that \(u \in S^1([-T, T] \times \mathbb{R}^N)\) and
\[
\|u\|_{S^1([-T, T] \times \mathbb{R}^N)} \leq C(E, M),
\]
where \(E\) and \(M\) denote the mass and energy of the system. In other words, the existence time \(T\) of the solution for IVP \eqref{IVP2} will depend solely on the norm \(\|u_0\|_{H^1_x(\mathbb{R}^N)}\).

\begin{theorem}\label{goodlwp}
	Let \( 2 < N < 6 \), \(0 < b_1, b_2 < \min\{2, \frac{N}{2} , \frac{6-N}{2} \}\), \(0 < p_1 < \frac{4-2b_1}{N-2}\), and \(p_2 = \frac{4-2b_2}{N-2}\). Suppose 
	\begin{itemize}
		\item[i)] $\lambda_1$, $\lambda_2 >0$;
		\item[ii)] $ \lambda_1 <0$, $\lambda_2 >0$, $p_1 < p_2$ and $\frac{p_1b_2}{p_2} \leq b_1 \leq b_2$;
		\item[iii)] $ \lambda_1 <0$, $\lambda_2 >0$ and $b_2 < b_1 < \frac{N(p_2-p_1)}{p_2+2}$.
		\end{itemize}
	Given \(u_0 \in H^1_x\), there exists \(T = T(\|u_0\|_{H^1_x})\) sufficiently small such that the Cauchy problem \eqref{IVP2} has a unique strong solution in \(C([0,T], H^1_x(\mathbb{R}^N))\) satisfying
	\[
	\|u\|_{S^1([-T,T] \times \mathbb{R}^N)} \leq C(E, M).
	\]
\end{theorem}

\begin{proof}

By practice, consider \(\lambda_1 = \lambda_2 = 1\).

From Proposition \ref{lwp}, we may assume that there exists a strong solution \(u\) to \eqref{IVP2} on the slab \([-T,T] \times \mathbb{R}^{N}\) and show that \(u\) has a finite \(X^1_{b_1, b_2}\) bound on this slab as long as \(T = T(\|u_0 \|_{H^{1}_{x}})\) is sufficiently small.

Let \(w\) be the unique strong global solution to the energy-critical equation \eqref{IVP} with initial data \(w_0 = u_0\) at time \(t = 0\). 
From Strichartz estimates we have $\|\nabla e^{it\Delta}w_0\|_{W([-T,T])} \leq \|\nabla w_0\| < \infty$. So, we can consider $T$ sufficiently small such that $\|\nabla e^{it\Delta} w_0\|_{W([-T,T])} $ also is small and by Theorem \ref{lwpECnls} we have that \(\| \nabla w \|_{S(L^{2}, [-T, T])} \) is bounded.  

Then, as $X_{b_1, b_2}$-norm is admissible, we can subdivide \([0,T]\) into \(J = J(E, \eta)\) subintervals \(I_j = [t_j, t_{j+1}]\) such that
\begin{equation}\label{estimate1}
	\|\nabla w\|_{X_{b_1, b_2} (I_j)} \sim \eta
\end{equation}
for some small \(\eta\) to be specified later. 

The nonlinear evolution of \(w\) being small on the interval \(I_j\) implies that the linear evolution is also small on \(I_j\). In fact, by using Strichartz estimates and Corollary (\ref{cor1}), we obtain the following inequality
\begin{equation*}
	\begin{split}
		\|\nabla [e^{i(t-t_j) \Delta} w(t_j)]\|_{X_{b_1, b_2} (I_j)} \leq & \; \|\nabla w\|_{X_{b_1, b_2} (I_j)} + \||x|^{-b_2}|w|^{p_2}w\|_{2, \frac{2N}{N+2}} \\
		\leq & \; \eta + C \; \|\nabla w\|_{X_{b_1, b_2} (I_j)}^{p_2 +1} \leq  \eta + C \eta^{p_2 +1}
	\end{split}
\end{equation*}
where \(C\) is an absolute constant that depends on the Strichartz constant. By choosing \(\eta\) sufficiently small, for any \(0 \leq j \leq J\), we obtain the following estimate

\begin{equation}\label{estimate2}
	\| \nabla[ e^{i(t-t_j) \Delta}  w(t_j)]\|_{X_{b_1, b_2} (I_j)} \leq 2 \eta.
\end{equation}
To estimate \(u\) on the interval \(I_0\), recalling that \(u(0) = w(0) = u_0\), we use Strichartz estimates and Corollary \(\ref{nlestimates2}\):
\begin{equation*}
	\begin{split}
		\|\nabla u\|_{X_{b_1, b_2} (I_0)} \leq & \; \|\nabla [e^{it \Delta}  w_0]\|_{X_{b_1, b_2} (I_0)} + C \; |I_0|^{1 - \frac{p_1+2}{\gamma_{b_1}}} \| \nabla u\|_{X_{b_1, b_2} (I_0)}^{p_1+1} + \| \nabla u\|_{X_{b_1, b_2} (I_0)}^{p_2+1}  \\
		\leq & \;  2 \eta + C \; T^{1 - \frac{p_1+2}{\gamma_{b_1}}} \| \nabla u\|_{X_{b_1, b_2} (I_0)}^{p_1+1} + \|\nabla u\|_{X_{b_1, b_2} (I_0)}^{p_2+1}.
	\end{split}
\end{equation*}

Assuming that $\eta$ and $T$ are sufficiently small, a standard continuity argument then yields
\begin{equation}\label{estimate3}
	\| \nabla  u\|_{X_{b_1, b_2} (I_0)} \leq 4  \eta.
\end{equation}
Thus, by Sobolev embedding, we get
\[
\|u\|_{Z(I_0)} \leq C\|\nabla u \|_{W^0(I_0)} \leq 4C \eta.
\]
By definition of the $X_{b_1, b_2}$-norm, \eqref{boundedhipL} holds on $I=I_0$ for $L = 4 C \eta$. 
Moreover, from Lemma \ref{boundedINLScritical} and Corollary \ref{benergycinetic}, we have
\[
\|\nabla w\|_{L_t^{\infty}L_x^2} \leq C(\eta) \|\nabla w_0\|_{L_x^2} \leq C( E(u_0), M(u_0)).
\]
Therefore, we have proved that \eqref{energybound1-L} holds with $E= C(E(u_0), M(u_0))$. Also, as \eqref{energybound2-L} holds with $E'= 0$, we are in the position to apply the stability result from Theorem \ref{LTP}, provided the error, which in this case is the first nonlinearity, is sufficiently small. By Corollary \ref{nlestimates2},
\[
\|\nabla (|x|^{-b_1}|u|^{p_1}u)\|_{S'(L^{2}, I_0)} \lesssim T^{1- \frac{p_1+2}{\gamma_{b_1}}} \|\nabla u\|_{X_{b_1, b_2} (I_0)}^{p_1+1} \lesssim T^{1- \frac{p_1+2}{\gamma_{b_1}}} (4 \eta)^{p_1+1},
\]
where in the last inequality we use \eqref{estimate3}. Thus, choosing $T$ sufficiently small (depending only on the energy and the mass of the initial data), we get
\begin{equation}\label{boundedit1}
	\|\nabla (|x|^{-b_1}|u|^{p_1}u)\|_{S'(L^{2}, I_0)} \leq \varepsilon,
\end{equation}
where $\varepsilon = \varepsilon(E(u_0), M(u_0))$ is a small constant to be chosen later. Taking $\varepsilon$ sufficiently small, the hypotheses of Theorem \ref{LTP} are satisfied, which implies that the conclusion holds. In particular,
\begin{equation}\label{boundedit2}
	\|\nabla (u - w)\|_{S(L^2,I_0)} \leq \, C(E,M) \, \varepsilon,
\end{equation}
for a small positive constant $C$ that depends only on the dimension $N$.
By Strichartz estimates, \eqref{boundedit1} and \eqref{boundedit2} imply
\begin{equation}\label{boundedit3}
	\|u(t_1) - w(t_1)\|_{\dot{H}^1_x} \lesssim \, C(E,M) \,  \varepsilon ,
\end{equation}
\begin{equation}\label{boundedit4}
	\|\nabla [e^{i(t-t_1)\Delta}(u(t_1) - w(t_1))]\|_{X_{b_1, b_2}} \lesssim \,  C(E,M) \,  \varepsilon .
\end{equation}
Using \eqref{estimate2}, \eqref{boundedit3}, \eqref{boundedit4}, Corollary \ref{nlestimates2}, and Strichartz estimates, we estimate
\begin{equation*}
	\begin{split}
		\| \nabla u\|_{X_{b_1, b_2}(I_1)} \leq & \|\nabla [e^{i(t-t_1)\Delta}u(t_1)]\|_{X_{b_1, b_2}(I_1)} + CT^{1- \frac{p_1+2}{\gamma_{b_1} }}\|\nabla u\|_{X_{b_1, b_2}(I_1)}^{p_1+1} + C\| \nabla u\|_{X_{b_1, b_2}(I_1)}^{p_2+1} \\
		\leq & \| \nabla [e^{i(t-t_1)\Delta}w(t_1)]\|_{X_{b_1, b_2}(I_1)} + \| \nabla [e^{i(t-t_1)\Delta}(u(t_1)- w(t_1))]\|_{X_{b_1, b_2}(I_1)}\\
		& +  CT^{1- \frac{p_1+2}{\gamma_{b_1} }}\|\nabla u\|_{X_{b_1, b_2}(I_1)}^{p_1+1} + C\|\nabla u\|_{ X_{b_1, b_2}(I_1)}^{p_2+1}\\
		\leq & 2 \eta + C(E,M) \, \varepsilon +  CT^{1- \frac{p_1+2}{\gamma_{b_1} }}\|\nabla u\|_{X_{b_1, b_2}(I_1)}^{p_1+1} + C\|\nabla  u\|_{X_{b_1, b_2}(I_1)}^{p_2+1}.
	\end{split}
\end{equation*}
A standard continuity argument then yields
\[
\|\nabla u\|_{X_{b_1, b_2}(I_1)} \leq 4 \eta,
\]
provided $\varepsilon$ is chosen sufficiently small depending on $E$ and $M$, which amounts to taking $T$ sufficiently small depending on $E$ and $M$. Thus, \eqref{estimate3} holds with $I_0$ replaced by $I_1$, and we are again in the position to apply Theorem \ref{LTP} on $I=I_1$ to obtain
\[
\|\nabla (u-w)\|_{S(L^2 , I_1)} \leq \, C(E, M) \, \varepsilon^{2}.
\]
By induction, for every $0 \leq j \leq J$, we obtain
\begin{equation}\label{boundedit5}
	\| \nabla u\|_{X_{b_1, b_2}(I_j)} \leq 4 \eta,
\end{equation}
provided $\epsilon$ (and hence $T$) is sufficiently small depending on the energy and the mass of the initial data. Adding  \eqref{boundedit5} over all $0 \leq j \leq J$, we get
\begin{equation}\label{boundedit6}
	\| \nabla u\|_{X_{b_1, b_2}([0,T])} \leq 4 J \eta = C(E).
\end{equation}
Now, to prove that $\|u\|_{S^1(L^2,I)}  \leq C(E, M)$, we proceed as follows. 
By Strichartz estimates, Corollary \ref{nlestimates2}, Corollary \ref{benergycinetic}, \eqref{boundedit6}, and recalling that $T = T(E, M)$, we obtain
\begin{equation}\label{boundedit7}
	\|\nabla u\|_{S(L^2,[0,T])} \lesssim \|u_0\|_{\dot{H}^1_x} + T^{1- \frac{p_1+2}{\gamma_{b_1}}}\| \nabla u\|_{X_{b_1, b_2}([0,T])}^{p_1+1} + \|\nabla u\|_{X_{b_1, b_2}([0,T])}^{p_2+1} \leq C(E, M).
\end{equation}
Similarly,
\begin{equation}\label{boundedit8}
	\begin{split}
		\| u\|_{S(L^2,[0,T])} \lesssim & \, \|u_0\|_{L^2_x} + T^{1- \frac{p_1+2}{\gamma_{b_1}}}\|\nabla  u\|_{X_{b_1, b_2}([0,T])}^{p_1}\| u\|_{X_{b_1, b_2}([0,T])} \\
		& + \|\nabla  u\|_{X_{b_1, b_2}([0,T])}^{p_2}\|u\|_{X_{b_1, b_2}([0,T])}.
	\end{split}
\end{equation}
Subdividing $[0,T]$ into $N_1 \sim N_1(E, M, \delta)$ subintervals $J_k$ such that
$$
\| \nabla u \|_{X_{b_1, b_2}(J_k)} \sim \delta
$$
for some small constant $\delta > 0$, the computations that lead to \eqref{boundedit8} now yield
$$
\| u\|_{S(L^2, J_k)} \lesssim \,  M^{\frac{1}{2}} +  C(E,M) \, \delta^{p_1} \| u\|_{X_{b_1, b_2}(J_k)} + \delta^{p_2} \| u\|_{X_{b_1, b_2}(J_k)}.
$$
Choosing $\delta$ sufficiently small depending on $E$ and $M$, we obtain
$$
\| u\|_{S(L^2, J_k)} \leq \, C(E,M)
$$
on every subinterval $J_k$. Adding these bounds over all subintervals $J_k$, we get
$$
\| u\|_{S(L^2, [0,T])} \leq \, C(E,M).
$$
This concludes the proof.
\end{proof}
\begin{proof}[Proof of Theorem \ref{gwp}]
Given an initial condition \(u_0 \in H^1_x(\mathbb{R}^N)\), our objective is to prove that the corresponding solution \(u\) for the initial value problem (IVP) \eqref{IVP2} extends for all time. Suppose, by contradiction, that the maximum existence time \(T_{\max}\) for the solution \(u\) is finite.

Define \(T = T(\|u_0\|_{H^1}) > 0\) based on Theorem \ref{goodlwp}, where \(\|u\|_{S^1(L^2,[0,T])} \leq C(E,M)\), and \(C(E,M) > 0\) depends only on \(M = M[u_0]\) and \(E[u_0]\). As a reminder, the Corollary ensures uniform control of the kinetic energy of \(u\), i.e., \(\|u(t)\|_{H^1} \leq C(E,M)\) for all \(t \in (-T_{\\min}, T_{\max})\). Now, divide the interval \([0, T_{\max} - \varepsilon]\) into disjoint intervals \(J_k\) with lengths smaller than \(T\). By repeatedly applying Theorem \ref{goodlwp} to each interval \(J_k\), we deduce by the uniqueness of the solution, conservation of mass, and uniform control of kinetic energy that \(\|u\|_{S^1(L^2,[0,T_{\max}-\varepsilon])} \leq LC(E,M)\). Applying Proposition \ref{Blow-upcriterion}, we find a positive \(\delta\) such that \(u\) remains defined on \([0, T_{\max} - \varepsilon + \delta]\). By choosing \(\varepsilon < \delta\), we arrive at a contradiction. Consequently, \(u\) is a global solution.

\end{proof}
\section{Finite time Blowup}
In this session, we establish Theorem \ref{blwupresult}, which addresses the existence of blow-up solutions for the nonlinear Schrödinger (INLS) equation with two power-type nonlinearities. Our investigation begins with the virial identity, a fundamental concept that connects energy distribution to solution behavior.
\begin{lemma}\label{BLfinitelemma}
	Let $3\leq  N < 6$, $\lambda_2 < 0$, $0 < p_1 < \frac{4-2b_1}{N-2}$, $p_2 = \frac{4-2b_2}{N-2}$, and $0 < b_1, b_2 < \min \{2, \frac{6-N}{2}\}$. Consider the variance
	\begin{equation*}
		V(t) = \int_{\mathbb{R}^{N}} |x|^2 |u(x,t)|^2 \, dx,
	\end{equation*}
	where $u$ is the maximal strong solution of \eqref{IVP2}. If $u_0 \in \Sigma$, then $V \in C^2(-T_{\min}, T_{\max})$, and as a function of $t > 0$, $V$ is decreasing and concave in each of the following cases:
	\begin{enumerate}[(a)]
		\item $\lambda_1 > 0$, $p_1 < p_2$, $b_1 < b_2$, and $E(u) < 0$,
		\item $\lambda_1 > 0$, $b_1 < \frac{N(p_2-p_1)}{p_2+2}$, and $E(u) < 0$,
		\item $\lambda_1 < 0$, $\frac{4-2b_1}{N} < p_1 < p_2$, $b_1 < b_2$, and $E(u) < 0$,
		\item $\lambda_1 < 0$, $\frac{4-2b_1}{N} < p_1$, $b_1 < \frac{N(p_2-p_1)}{p_2+2}$, and $E(u) < 0$,
		\item $\lambda_1 < 0$, $p_1 < p_2$, $\frac{p_1b_2}{p_2} \leq b_1 \leq b_2$, and $E(u_0) + CM(u_0) < 0$ for some suitably large constant $C$ (depending as usual on $N$, $p_1$, $p_2$, $\lambda_1$, $\lambda_2$).
	\end{enumerate}
\end{lemma}

\begin{proof}
As well as in \cite{farah16}, we can easily prove that for all \(t \in (-T_{\min}, T_{\max})\):
\begin{equation}\label{derivativevariance}
	V'(t) = -4y(t)
\end{equation}
where
\[
y(t) = -\Im\displaystyle \int_{\mathbb{R}^{N}} \overline{u}(x,t)\left( \nabla u(x,t) \cdot x\right)\, dx
\]
and
\begin{equation}\label{virialid}
	V''(t) = -4y'(t) = 8 \|\nabla u(t)\|^2 + \frac{4\lambda_1(Np_1 + 2b_1)}{p_1 +2 } \, \|u(t)\|_{L_{b_1}^{p_1+2}}^{p_1+2} + \frac{4\lambda_2(Np_2 + 2b_2)}{p_2 +2 } \, \|u(t)\|_{L_{b_2}^{p_2+2}}^{p_2+2}.
\end{equation}

Hence, it is sufficient to show that
\begin{equation}\label{ineqblow}
	y'(t) \geq c\|\nabla u\|^2 > 0
\end{equation}
for a small positive constant \(c\).

Suppose that ocours (a). From \eqref{virialid} and \eqref{energy}, it follows
\begin{equation*}
\begin{split}
y'(t) = & -2  \|\nabla u(t)\|^2 -  \frac{\lambda_1(Np_1 + 2b_1)}{p_1 +2 }  \|u(t)\|_{L_{b_1}^{p_1+2}}^{p_1+2} - \frac{\lambda_2(Np_2 + 2b_2)}{p_2 +2 } + \|u(t)\|_{L_{b_2}^{p_2+2}}^{p_2+2} +(Np_2 +2b_2) E(u(t)) \\
& - (Np_2 +2b_2)  E(u(t))\\
\geq & \frac{Np_2 -4 + 2b_2}{2}\|\nabla u(t)\|^2 + \frac{\lambda_1}{p_1+2}[N(p_2-p_1) + 2(b_2 - b_1)]\|\nabla u(t)\|_{L_{b_1}^{p_1+2}}^{p_1+2}  \geq \frac{Np_2 -4 + 2b_2}{2}\|\nabla u(t)\|^2,
\end{split}
\end{equation*}
where in the first inequality we use the hypothesis $E(u(t)) <0$ and in the last inequality we use the conditions under the parameters $\lambda_1$, $\lambda_2$, $b_1$, $b_2$, $p_1$ and $p_2$. Thus, \eqref{ineqblow} follows with $c=  \frac{Np_2 -4 + 2b_2}{2}$.\\

Suppose that condition (a) and (b) holds. From \eqref{virialid} and \eqref{energy}, it follows

\begin{equation*}
	\begin{split}
		y'(t) = & -2  \|\nabla u(t)\|^2 -  \frac{\lambda_1(Np_1 + 2b_1)}{p_1 +2 }  \|u(t)\|_{L_{b_1}^{p_1+2}}^{p_1+2} - \frac{\lambda_2(Np_2 + 2b_2)}{p_2 +2 } + \|u(t)\|_{L_{b_2}^{p_2+2}}^{p_2+2} +(Np_2 +2b_2) E(u(t)) \\
		& - (Np_2 +2b_2)  E(u(t)) \\
		\geq & \frac{Np_2 -4 + 2b_2}{2}\|\nabla u(t)\|^2 + \frac{\lambda_1}{p_1+2}[N(p_2-p_1) + 2(b_2 - b_1)]\|\nabla u(t)\|_{L_{b_1}^{p_1+2}}^{p_1+2}  \geq \frac{Np_2 -4 + 2b_2}{2}\|\nabla u(t)\|^2,
	\end{split}
\end{equation*}

where in the first inequality we use the hypothesis \(E(u(t)) < 0\), and in the last inequality, we rely on the conditions involving the parameters \(\lambda_1\), \(\lambda_2\), \(b_1\), \(b_2\), \(p_1\), and \(p_2\). Thus, \eqref{ineqblow} follows with \(c = \frac{Np_2 - 4 + 2b_2}{2}\).

Now, if conditions (c) and (d) hold, then from \eqref{virialid}, we have
\begin{equation*}
	\begin{split}
		y'(t) = & -2  \|\nabla u(t)\|^2 -  \frac{\lambda_1(Np_1 + 2b_1)}{p_1 +2 }  \|u(t)\|_{L_{b_1}^{p_1+2}}^{p_1+2} - \frac{\lambda_2(Np_2 + 2b_2)}{p_2 +2 } \|u(t)\|_{L_{b_2}^{p_2+2}}^{p_2+2} \\
		= &  -2  \|\nabla u(t)\|^2 +  (Np_1 + 2b_1) \left( \frac{1}{2} \|\nabla u(t)\|^2 + \frac{\lambda_2}{p_2 + 2}\|u(t)\|_{L_{b_2}^{p_2+2}}^{p_2+2} - E(u(t)) \right) - \frac{\lambda_2(Np_2 + 2b_2)}{p_2 +2 } \|u(t)\|_{L_{b_2}^{p_2+2}}^{p_2+2} \\
		\geq & \frac{Np_1 -4 + 2b_1}{2}\|u(t)\|^2 + \frac{\lambda_2}{p_2+2}[N(p_1-p_2) + 2(b_1 - b_2)]\|u(t)\|_{L_{b_2}^{p_2+2}}^{p_2+2}  \geq \frac{Np_1 -4 + 2b_1}{2}\|u(t)\|^2.
	\end{split}
\end{equation*}

Therefore, we conclude that \eqref{ineqblow} holds with \(c = \frac{Np_1 -4 + 2b_1}{2}\).

Finally, consider the case where condition (e) holds. Since \(p_2 > \frac{4 - 2b_2}{N}\), we can find a sufficiently small \(\varepsilon\) such that \(p_2 > \frac{2(2 - b_2 + \varepsilon)}{N}\). It is immediate that \(\theta := \frac{2(2 + \varepsilon)}{Np_2 + 2b_2} < 1\). Therefore, by adding and subtracting \(\frac{\lambda_2(Np_2 - 2b_2)\theta}{p_2 + 2}\) to \(y'(t)\) and using the energy expression, we obtain

\begin{equation*}
	\begin{split}
		y'(t) = & -2  \|\nabla u(t)\|^2 -  \frac{\lambda_1(Np_1 + 2b_1)}{p_1 + 2}  \|u(t)\|_{L_{b_1}^{p_1+2}}^{p_1+2} - \frac{\lambda_2(Np_2 + 2b_2)(1 - \theta)}{p_2 + 2} \|u(t)\|_{L_{b_2}^{p_2+2}}^{p_2+2} \\
		& - \frac{\lambda_2(Np_2 + 2b_2) \theta }{p_2 + 2} \|u(t)\|_{L_{b_2}^{p_2+2}}^{p_2+2} \\
		\geq & -2  \|\nabla u(t)\|^2 -  \frac{\lambda_1(Np_1 + 2b_1) \theta}{p_1 + 2}  \|u(t)\|_{L_{b_1}^{p_1+2}}^{p_1+2}  - \frac{\lambda_2(Np_2 + 2b_2)(1 - \theta)}{p_2 + 2} \|u(t)\|_{L_{b_2}^{p_2+2}}^{p_2+2} \\
		& + (Np_2 + 2b_2) \theta \left[ \frac{1}{2} \| \nabla u\|^2 + \frac{\lambda_1}{p_1 + 2} \| u \|_{L_{b_1}^{p_1+2}}^{p_1+2} - E(u_0) \right] \\
		\geq & \left(- 2 + \frac{(Np_2 + 2b_2)\theta}{2} \right)\|\nabla u\|^2 + \frac{\lambda_1 \theta}{p_1+2}\left[ (Np_2+2b_2) - (Np_1+2b_1) \right] \,  \|u(t)\|_{L_{b_1}^{p_1+2}}^{p_1+2} - (Np_2 + 2b_2) \theta E(u_0) \\
		& - \frac{\lambda_2(Np_2 + 2b_2)(1- \theta )}{p_2 +2 } \|u(t)\|_{L_{b_2}^{p_2+2}}^{p_2+2} \\
		= & \varepsilon \,  \| \nabla u\|^2  + \frac{\lambda_1 \theta}{p_1+2} \left[ (N(p_2-p_1) - 2(b_2 - b_1) \right] \,  \|u(t)\|_{L_{b_1}^{p_1+2}}^{p_1+2} - (Np_2 + 2b_2) \theta E(u_0) \\
		& - \frac{\lambda_2(Np_2 + 2b_2)(1- \theta )}{p_2 +2 } \|u(t)\|_{L_{b_2}^{p_2+2}}^{p_2+2}.
	\end{split}
\end{equation*}

By Proposition \ref{bweightednorm}, we have
\begin{equation*}
	\begin{split}
		\frac{|\lambda_1|\theta}{p_1+2} \left[ N(p_2 - p_1) + 2(b_2 - b_1) \right]\, \|u(t)\|_{L_{b_1}^{p_1+2}}^{p_1+2} \leq & \frac{|\lambda_1| \theta}{p_1+2} \left[ N(p_2 - p_1) + 2(b_2 - b_1) \right]\eta \|u(t)\|_{L_{b_2}^{p_2+2}}^{p_2+2} \\
		& + \frac{|\lambda_1| \theta}{p_1+2} \left[ N(p_2 - p_1) + 2(b_2 - b_1) \right] C(\eta, M(u_0))
	\end{split}
\end{equation*}
for any \(\eta > 0\). Choosing \(\eta\) sufficiently small such that
\[
\frac{|\lambda_1| \theta}{p_1+2} \left[ N(p_2 - p_1) + 2(b_2 - b_1) \right] \eta \leq \frac{|\lambda_2|(Np_2 + 2b_2)(1- \theta )}{\lambda_1 (p_2 +2)}
\]
results in
\begin{equation*}
	\begin{split}
		y'(t) \geq \varepsilon \, \| \nabla u\|^2 - C(\lambda_i, p_i, b_i, N, M(u_0)) - (Np_2 + 2b_2) \theta E(u_0)
	\end{split}
\end{equation*}
for \(i=1,2\). As long as
\[
(Np_2 + 2b_2) \theta E(u_0) + C(\lambda_i, p_i, b_i, N, M(u_0)) < 0
\]
we have
\[
y'(t) \geq \varepsilon \, \| \nabla u\|^2.
\]
\end{proof}

The last lemma ensures the proof of Theorem \ref{blwupresult} as we will see below.

\begin{proof}[Proof of Theorem \ref{blwupresult}]
	Using the notation from the previous lemma, we know that by hypothesis, \(y(0) = y_0 > 0\). Furthermore, from \eqref{ineqblow}, we deduce that \(y(t) > y(0) > 0\) for all times \(t > 0\). Additionally, by Hölder's inequality, for all \(t \in (-T_{\min}, T_{\max})\):
	
	\[
	y(t) \leq \|xu(t)\| \|\nabla u(t)\|
	\]
	
	Consequently, we have
	
	\begin{equation}\label{EDOBLOw}
		\|\nabla u(t)\| \geq \frac{y(t)}{\|xu(t)\|}
	\end{equation}
	
	Combining \eqref{ineqblow} and \eqref{EDOBLOw}, we obtain a first-order ordinary differential equation for \(y\)
	
	\begin{equation*}
		\begin{cases}
			y'(t) \geq c \frac{y^2(t)}{\|xu(t)\|}, \\
			y(0) = y_0 > 0.
		\end{cases}
	\end{equation*}
	
	This implies that there exists a time \(T^{*} \leq \frac{\|xu_0\|^2}{cy_{(0)}}\) such that:
	
	\[
	\lim_{t \rightarrow T^{*}} \|\nabla u \| = \infty.
	\]
\end{proof}

\end{document}